\documentclass[reqno]{amsart}
\usepackage[colorlinks=true]{hyperref}
\usepackage{amssymb,amsfonts,amsmath,amsthm,mathrsfs}
\usepackage{cite}   
\usepackage{amsaddr}   

\title[Domain perturbation for parabolic equations]
      {An abstract approach to domain perturbation for parabolic equations
        and parabolic variational inequalities}
\author{Parinya Sa Ngiamsunthorn}
\address{School of Mathematics and Statistics\\
  University of Sydney, NSW 2006 Australia}
\email{pasa4391@uni.sydney.edu.au}


\theoremstyle{plain}
\newtheorem{theorem}{Theorem}[section]
\newtheorem{proposition}[theorem]{Proposition}

\newtheorem{lemma}[theorem]{Lemma}

\theoremstyle{definition}
\newtheorem{definition}[theorem]{Definition}

\theoremstyle{remark}
\newtheorem{remark}[theorem]{Remark}

\DeclareMathOperator{\supp}{supp}

\begin{document}
\date{March 10, 2011}
\begin{abstract}
We study the behaviour of solutions of linear non-autonomous parabolic equations subject to Dirichlet 
or Neumann boundary conditions under perturbation of the domain. 
We prove that Mosco convergence of function spaces for non-autonomous parabolic problems is equivalent to Mosco
convergence of function spaces for the corresponding elliptic problems.
As a consequence, we obtain convergence of solutions of non-autonomous parabolic equations under domain perturbation 
by variational methods using the same characterisation of domains as in elliptic case.
A similar technique can be applied to obtain convergence of weak solutions of parabolic variational inequalities when
the underlying convex set is perturbed. 
\end{abstract}
%

\maketitle

\section{Introduction}
\label{sec:intro}

The primary aim of this paper is to study convergence properties 
of solutions of linear \emph{non-autonomous} parabolic equations under perturbation of the domain. 
We consider a sequence of bounded open
sets $(\Omega_n)_{n=1}^\infty$ in $\mathbb R^N$ converging to a bounded
open set $\Omega$ and investigate the behaviour of solutions of 
the following parabolic equations 
\begin{equation}
\label{eq:paraIntro}
  \left\{ 
  \begin{aligned}
    \frac{\partial u}{\partial t} + \mathcal A_n(t) u &= f_n(x,t) &&\quad \text { in } \Omega_n \times (0,T]\\
    \mathcal B_n(t) u &= 0  &&\quad \text{ on } \partial \Omega_n \times (0,T]\\
     u(\cdot,0) &= u_{0,n} &&\quad \text{ in } \Omega_n , \\
  \end{aligned} 
  \right . 
\end{equation}
where $\mathcal A_n$ is an elliptic operator of the form
\begin{displaymath}
 \mathcal A_n(t) u := - \partial_i[a_{ij}(x,t) \partial_j u + a_i(x,t) u]
                     + b_i(x,t) \partial_i u  + c_0(x,t) u,
\end{displaymath}
and $\mathcal B_n(t)$ is one of the following boundary conditions
\begin{displaymath}
\begin{aligned}
 \mathcal B_n(t) u &:= u &&\quad \text{Dirichlet boundary condition} \\ 
 \mathcal B_n(t) u &:= [a_{ij}(x,t) \partial_j u + a_i(x,t)u] \;\nu_i &&\quad
 \text{Neumann boundary condition.}
\end{aligned}
\end{displaymath} 
In abstract form, \eqref{eq:paraIntro} can be written as
\begin{equation}
\label{eq:paraAbsIntro}
 \left \{
 \begin{aligned}
 u'(t) + A_n(t) u &= f(t)  \quad \text{ for } t \in (0,T]\\
 u(0) &= u_{0,n}  
 \end{aligned}
 \right .
\end{equation} 
in a Banach space $V_n$, where $V_n := H^1_0(\Omega_n)$ for Dirichlet problems or $V_n := H^1(\Omega_n)$ for Neumann problems.
We refer to Section \ref{sec:prelim} for the precise framework of these parabolic equations.
We are particularly interested in singular domain perturbation 
so that change of variables is not possible on these domains. 
Typically, the common examples include a sequence of dumbbell shape domains 
with shrinking handle, and a sequence of domains with cracks.  
Moreover, we mostly do not assume any smoothness of $\Omega_n$ 
and $\Omega$.
The second aim of this paper is to study a similar convergence properties
of solutions of parabolic variational inequalities 
\begin{equation}
\label{eq:parVarIneqKIntro}
 \left\{ 
  \begin{aligned}
   \langle u'(t), v-u(t) \rangle  + \langle A(t)u(t), v-u(t) \rangle 
     - \langle f(t) , v - u(t) \rangle &\geq 0, \quad \forall v \in K \\
   u(0) &= u_0.
   \end{aligned}
  \right .
\end{equation}
on $(0,T)$ when we perturb the underlying convex set $K$ in the problem.

To deal with non-autonomous parabolic equations, it is common to apply variational methods. 
In this paper we prove that under suitable assumptions on domains, a sequence of solutions $u_n$ of 
\eqref{eq:paraIntro} converges to the solution $u$ of a linear non-autonomous 
parabolic equation on the limit domain $\Omega$ (\eqref{eq:paraIntro} with $n$ deleted). 
This result sometimes refers to \emph{stability} of solutions under
domain perturbation or \emph{continuity} of solutions with respect to the domain. 
The method presented in this work is rather an abstract approach. In particular, it 
can be applied to obtain stability of solutions under domain perturbation for both Dirichlet problems
(Theorem \ref{th:unifConvSolDir}) and Neumann problems 
(Theorem \ref{th:unifConvSolNeu}). In general, it is more difficult when handling Neumann boundary condition. 
We cannot simply consider the trivial extension by
zero outside the domain for functions in the sobolev space $H^1(\Omega)$ because the
extended function does not belong to $H^1(\mathbb R^N)$. Moreover there
is no smooth extension from  $H^1(\Omega)$ to $H^1(\mathbb R^N)$ as we
do not impose any regularity of the domain. This means the compactness result for a sequence of solutions $u_n$ in 
\cite[Lemma 2.1]{MR1404388} cannot be applied in the case of Neumann
problems. However, our abstract approach can be applied to Neumann problems.
We refer to Section \ref{sec:stabilityParaEq} for the study on stability of solutions of 
non-autonomous parabolic equations under domain perturbation. 

The key result that enables us to determine a sufficient condition on domains 
for which solutions converge under domain perturbation is Theorem
\ref{th:equiMosco}. In particular, Theorem \ref{th:equiMosco} shows that continuity of
solutions for non-autonomous problems can be deduced from the
corresponding elliptic problems via \emph{Mosco convergence}.
We refer to Section \ref{sec:moscoConv} for the definition and results on Mosco convergence. 
A similar deduction is well-known for autonomous parabolic equations.
This is simply because we can apply semigroup methods together with convergence result of 
degenerate semigroups due to Arendt \cite[Theorem 5.2]{MR1832168}.
In  Section 6 of the same paper, stability of solutions of Dirichlet heat equation 
is given as an example. Further examples on other boundary conditions
including Neumann and Robin boundary conditions
can be found in \cite[Section 6]{MR2119988}.  Indeed, for quasilinear parabolic equations, 
Simondon \cite{MR1855977} also obtained continuity of
solutions of parabolic equations under Dirichlet
boundary condition using a similar equivalence of Mosco convergences between
certain Banach spaces. However, Theorem \ref{th:equiMosco}
can be seen as an abstract generalisation of \cite{MR1855977}.
We show equivalence between Mosco convergences of various \emph{closed and convex subsets} of a
Banach space rather than Mosco convergences of a particular choice of closed subspaces of a
Banach space. The obvious reason
for this generalisation is that Mosco convergence was originally introduced in
\cite{MR0298508} for convex sets and was the main tool to establish
convergence properties of solutions of elliptic variational inequalities when the 
convex set is perturbed. The second advantage of Theorem
\ref{th:equiMosco} is that we do not only show the equivalence between
Mosco convergence of convex subsets of the Bochner-Lebesgue space $L^2((0,T),V)$ and
Mosco convergence of convex subsets of the corresponding Banach space
$V$  but also show that they are equivalent to Mosco convergence of
convex subsets of the Bochner-Sobolev spaces $W((0,T),V,V')$. Hence a similar technique can be
applied to obtain stability of solutions of parabolic variational inequalities when the underlying
convex set is perturbed (Theorem \ref{th:strongConvSolParVar}). We study convergence of solutions
of parabolic variational inequalities in Section \ref{sec:appParVar}.

An important consequence of Theorem \ref{th:equiMosco} is that the same conditions for a sequence of domains give
stability of solutions under domain perturbation for both parabolic
and elliptic equations. We refer to \cite{MR1822408, MR1951783, MR1995490, MR1955096} for the study of domain perturbation
for elliptic equations using Mosco convergence.

\section{Preliminaries on  parabolic equations and parabolic variational
inequalities}
\label{sec:prelim}

In this section we state some basic results on variational methods
for parabolic equations and a variational formulation for parabolic 
inequalities.

Suppose $V$ is a real separable and reflexive Banach space and $H$ is a
separable Hilbert space such that $V$ is dense in $H$. By identifying
$H$ with its dual space $H'$, we consider the following \emph{evolution triple}
\begin{displaymath}
 V \overset{d}{\hookrightarrow} H \overset{d}{\hookrightarrow} V'.
\end{displaymath}
Throughout this paper, we denote by $(\cdot|\cdot)$, the scalar 
product in $H$ and $\langle \cdot,\cdot \rangle$, the duality paring
between $V'$ and $V$. 
For an interval $(a,b) \subset \mathbb R$, 
we denote by $L^2((a,b),V)$ the Bochner-Lebesgue space. 
We define the Bochner-Sobolev space
\begin{displaymath}
W((a,b),V,V') := \{ u \in L^2((a,b),V) : u'\in L^2((a,b),V') \},
\end{displaymath}
where $u'$ is the derivative in the sense of distributions taking
values in $V'$. The space $W((a,b),V,V')$ is a Banach space when
equipped with the following norm
\begin{displaymath}
\|u\|_W := \Big (\int_a^b \|u(t)\|^2_V \;dt + \int_a^b \|u'(t)\|^2_{V'} \;dt
\Big )^{1/2}.
\end{displaymath}
It is well known that 
$W((a,b),V,V') \hookrightarrow  C([a,b],H)$, where the
space of $H$-valued continuous functions $C([a,b],H)$ equipped 
with the uniform norm (\cite[Theorem I1.3.1]{MR1156075}). Moreover, for $u,v \in W((a,b),V,V')$ and $a_0,
b_0 \in [a,b]$ with $a_0 < b_0$ we have the integration by parts formula
\begin{equation}
\label{eq:intByParts}
(u(b_0)|v(b_0))-(u(a_0)|v(a_0)) 
= \int_{a_0}^{b_0} \langle u'(t),v(t) \rangle + \langle v'(t),u(t) \rangle \;dt.
\end{equation}
Let $I,J$ be two sets, we write $J \subset \subset I$ if $\bar{J} \subset I$. 
For a subset $X$ of a Banach space $V$, we define the closed convex hull by
\begin{displaymath}
 \overline{\text{conv}} (X) := \overline{ \left \{ \sum_{i=1}^k \alpha_i x_i \mid x_i \in X,
                              \alpha_i \in \mathbb R, \alpha_i \geq 0, 
                               \sum_{i=1}^k \alpha_i =1, k =1,2, \ldots \right \}}.
\end{displaymath}

For each $t \in [0,T]$, suppose $a(t;\cdot,\cdot)$ is a continuous
bilinear form on $V$ satisfying the following hypothesis:
\begin{itemize}
 \item for every $u,v \in V$, the map $t \mapsto a(t;u,v)$ is
 measurable.
 \item there exists a constant $M >0$ independent of $t \in [0,T]$
 such that
   \begin{equation}
    \label{eq:biFormCont}
    |a(t;u,v)| \leq M \|u\|_V \|v\|_V ,
   \end{equation}
 for all $u,v \in V$.
 \item there exist $\alpha >0$ and $\lambda \in \mathbb R$ such that
   \begin{equation}
    \label{eq:biFormCoer}
    a(t;u,u) + \lambda \|u\|^2_H \geq \alpha \|u\|^2_V ,
   \end{equation}
 for all $u \in V$.
\end{itemize}
It follows that for each $t \in [0,T]$ and $u \in V$ the bilinear form
$a(t;\cdot,\cdot)$ induces a continuous linear operator
$A(t) \in \mathscr L(V,V')$ with
\begin{displaymath}
 \langle A(t)u, v \rangle = a(t;u,v) ,
\end{displaymath}
for all $u,v \in V$. We easily see from \eqref{eq:biFormCont} that
$\sup_{t \in [0,T]} \|A(t)\|_{\mathscr L(V,V')} \leq M$.

\subsection{Parabolic equations}
\label{subsec:varPara}
Let us consider the abstract parabolic equation
\begin{equation}
\label{eq:paraAbs}
  \left \{
  \begin{aligned}
   u'(t) + A(t) u &= f(t) \quad \text{ for } t \in (0,T] \\
   u(0) &= u_0 ,\\
  \end{aligned}
  \right .
\end{equation}
where $u_0 \in H$ and $f \in L^2((0,T),V')$.
A function $u \in W(0,T,V,V')$ satisfying \eqref{eq:paraAbs} 
is called a \emph{variational solution}.  
It is well known that $u$ is a variational solution of \eqref{eq:paraAbs}
if and only if $u \in L^2((0,T),V)$ and
\begin{equation}
\label{eq:solFormula}
 \begin{aligned}
  - \int_0^T (u(t)|v) \phi'(t) \;dt + \int_0^T a(t;u(t),v) \phi(t) \;dt \\
  = (u_0|v) \phi(0) + \int_0^T \langle f(t), v \rangle \phi(t) \;dt,
 \end{aligned}
\end{equation}
for all $v \in V$ and for all $\phi \in \mathscr D([0,T))$. 
The existence and uniqueness of solution is given in the following
theorem (see, for example, \cite[XVIII \S 3]{MR1156075} and \cite[\S
23.7]{MR1033497}). 
\begin{theorem}
\label{th:ExiUniParEq}
Given $f \in L^2((0,T),V')$ and $u_0 \in H$,
there exists a unique variational solution of \eqref{eq:paraAbs}
satisfies 
\begin{equation}
\label{eq:unifBoundSol}
 \|u \|_{W(0,T,V,V')} \leq C \Big (\|u_0 \|_H + \|f\|_{L^2((0,T),V')}
 \Big ).
\end{equation}
Moreover, if $\lambda =0$ in \eqref{eq:biFormCoer} the variational solution satisfies
\begin{equation}
\label{eq:boundOfSol}
 \|u(t)\|^2_H + \alpha \int_0^t \|u(s)\|^2_V \;ds
 \leq \|u_0\|^2_H + \alpha^{-1} \int_0^t \|f(s)\|^2_{V'} \;ds,
\end{equation}
for all $t \in [0,T]$.
\end{theorem}
Note that $v(t) := e^{-\lambda t}u(t)$ is a variational solution of \eqref{eq:paraAbs} with
$A(t)$ replaced by $A(t) + \lambda$. Hence
we can assume without loss of generality that $\lambda = 0$ in
\eqref{eq:biFormCoer}.

Let $\Omega$ be an open bounded set in $\mathbb R^N$.
Let $D \subset \mathbb R^N$ be a ball such that $\Omega \subset D$.
We shall consider a closed subspace $V$ of $H^1(\Omega)$ with
$H^1_0(\Omega) \subset V \subset H^1(\Omega)$. We take $H:= L^2(\Omega)$
and consider the evolution triple 
$ V \overset{d}{\hookrightarrow} H \overset{d}{\hookrightarrow} V'$.
In this paper, we study  bilinear forms $a(t;\cdot,\cdot)$
for $t \in [0,T]$ given by
\begin{equation}
\label{eq:biFormEx}
 a(t;u,v) := \int_{\Omega}[ a_{ij}(x,t) \partial_j u+ a_i(x,t)u] \partial_i v
                + b_i(x,t) \partial_i u v + c_0(x,t) uv \;dx,
\end{equation}
for $u,v \in V$. In the above, we use summation convention with $i,j$
running from $1$ to $N$. Also, we assume  $a_{ij}, a_i, b_i, c_0$ are
functions in 
$L^{\infty}(D \times (0,T))$ and there exists a constant $\alpha > 0$
independent of $(x,t) \in \Omega \times (0,T)$ such that
\begin{displaymath}
a_{ij}(x,t) \xi_i \xi_j \geq \alpha |\xi|^2 ,
\end{displaymath} 
for all $\xi \in \mathbb R^N$. It is clear that the map $t \mapsto a(t;u,v)$ is
measurable for all $u,v \in V$. Moreover, it can be verified that the form
$a(t;\cdot,\cdot)$ defined above satisfies \eqref{eq:biFormCont} and
\eqref{eq:biFormCoer} (see \cite{MR1156075}). 
Let $\mathcal A(t)$ be a differential operator on $V$ defined by
\begin{equation}
\label{eq:calAop}
 \mathcal A(t) u := - \partial_i[a_{ij}(x,t) \partial_j u + a_i(x,t) u]
                     + b_i(x,t) \partial_i u  + c_0(x,t) u .
\end{equation}
Given $u_0 \in L^2(D)$ and $f \in L^2(D \times (0,T))$, we consider the
following parabolic boundary value problem
\begin{equation}
\label{eq:para}
  \left\{ 
  \begin{aligned}
    \frac{\partial u}{\partial t} + \mathcal A(t) u &= f(x,t) &&\quad \text { in } \Omega \times (0,T]\\
    \mathcal B(t) u &= 0  &&\quad \text{ on } \partial \Omega \times (0,T]\\
     u(\cdot,0) &= u_0 &&\quad \text{ in } \Omega , \\
  \end{aligned} 
  \right . 
\end{equation}
where $\mathcal B(t)$ is one of the following boundary conditions
\begin{displaymath}
\begin{aligned}
 \mathcal B(t) u &:= u &&\quad \text{Dirichlet boundary condition} \\ 
 \mathcal B(t) u &:= [a_{ij}(x,t) \partial_j u + a_i(x,t)u] \;\nu_i &&\quad
 \text{Neumann boundary condition}
\end{aligned}
\end{displaymath}
It is well known that we can consider the boundary value problem 
\eqref{eq:para} as an abstract equation \eqref{eq:paraAbs}
by taking  $V = H^1_0(\Omega)$ for Dirichlet boundary
problem or $V=H^1(\Omega)$ for Neumann boundary problem 
(\cite[Corollary 23.24]{MR1033497}).

\subsection{Parabolic variational inequalities}
Suppose that $K$ is a closed and convex subset of $V$.
We denote by 
\begin{displaymath}
L^2((0,T), K) := \{ u \in L^2((0,T),V) \mid u(t) \in K \text{ a.e.}\}.
\end{displaymath}
For each $t \in (0,T)$, suppose $a(t; \cdot, \cdot)$ is a continuous 
bilinear form on $V$ satisfying \eqref{eq:biFormCont} and \eqref{eq:biFormCoer}.
As before, we denote the induced linear operator by $A(t)$.
Given $u_0 \in K$ and $f \in L^2((0,T),V')$, we wish to find $u$ such that
for a.e. $t \in (0,T)$, $u(t) \in K$ and 
\begin{equation}
\label{eq:parVarIneqK}
 \left\{ 
  \begin{aligned}
   \langle u'(t), v-u(t) \rangle  + \langle A(t)u(t), v-u(t) \rangle 
     - \langle f(t) , v - u(t) \rangle &\geq 0, \quad \forall v \in K \\
   u(0) &= u_0.
   \end{aligned}
  \right .
\end{equation}
A function $u \in W((0,T),V,V')$ satisfying \eqref{eq:parVarIneqK} is called a \emph{strong
  solution} of parabolic variational inequality \eqref{eq:parVarIneqK}. In this paper, we are
mainly interested in a weak formulation of the problem.
There are various (slightly different) definitions of weak solution of 
parabolic variational inequalities 
(see e.g.\cite{MR0428137},\cite{MR2210083},\cite{MR0259693}, \cite{MR0296479}).
We shall define a weak notion of solution similar to the one in
\cite{MR2210083} as follows.
\begin{definition}
\label{def:weakSolIneq}
A function $u$ is a \emph{weak solution} of parabolic
variational inequality \eqref{eq:parVarIneqK} if $u \in L^2((0,T),K)$ and
\begin{equation}
\label{eq:weakSolParVar}
\begin{aligned}
 \int_0^T \langle v'(t), v(t)-u(t) \rangle
   + \langle A(t)u(t), v(t)-u(t) \rangle  
   - \langle f(t), v(t)-u(t) \rangle \;dt\\
   + \frac{1}{2} \|v(0)-u_0\|^2_H
  \geq 0,
\end{aligned}
\end{equation}
for all $v \in W((0,T),V,V') \cap L^2((0,T),K)$.
\end{definition}
The existence and uniqueness of weak solutions of parabolic
variational inequalities have been studied by various authors 
according to their definitions. 
In our case, we can state the result in the following theorem.
\begin{theorem}
Given $u_0 \in K$ and $f \in L^2((0,T),V')$. There exists a unique weak
solution $u$ of the parabolic variational inequality \eqref{eq:parVarIneqK} satisfying $u \in
L^{\infty} ((0,T),H)$.
\end{theorem}
Note that the existence of our weak solution follows immediately from the 
existence results in \cite[Theorem 6.2]{MR0259693}. The uniqueness can be proved in the same way 
as in \cite[Theorem 2.3]{MR0296479}.
\section{Mosco convergence}
\label{sec:moscoConv}

We often consider Mosco convergence as introduced in \cite{MR0298508} when dealing with a sequence 
of functions belonging to a sequence of function spaces. 
In this section, we establish the key result which enables us to study 
domain perturbation for parabolic problems via the corresponding elliptic
problems. We prove that Mosco convergence of function spaces for 
non-autonomous parabolic problems is equivalent to Mosco convergence of function
spaces for the corresponding elliptic problems.
Throughout this section, we assume that
$V$ is a reflexive and separable Banach space, and $K_n, K$  are closed
and convex subsets of $V$. We start by giving a definition of Mosco convergence
in various spaces including $V$, $L^2((0,T),V)$ and $W((0,T),V,V')$. 
\begin{definition}
\label{def:moscoEll}
We say that $K_n$ converges to $K$ in the sense of Mosco if the following conditions hold
\begin{itemize}
 \item[$(M1)$] For every $u \in K$ there exists a sequence 
       $u_n \in K_n$ such that 
       $u_n \rightarrow u$ strongly in $V$.
 \item[$(M2)$] If $(n_k)$ is a sequence of indices converging to $\infty$, 
      $(u_k)$ is a sequence such that $u_k \in K_{n_k}$ for 
      every $k$ and 
      $u_n \rightharpoonup u$ weakly in $V$, then
      $ u \in K$.
\end{itemize}
\end{definition}
There is an alternative definition of Mosco convergence defined in terms
of Kuratowski limits. A general result on Mosco convergence and 
equivalence of these definitions can be found in \cite[Chapter 3]{MR773850}.

As discussed in Section \ref{sec:prelim}, solutions of 
parabolic equations and parabolic variational inequalities are functions in $L^2((0,T),V)$. 
Thus, it is worthwhile to study Mosco convergence in $L^2((0,T),V)$.
We denote by
\begin{displaymath}
L^2((0,T), K) := \{ u \in L^2((0,T),V) \mid u(t) \in K \text{ a.e.}\},
\end{displaymath}
and 
\begin{displaymath}
C([0,T],K) := \{ u \in C([0,T],V) \mid u(t) \in K \quad \forall t \in [0,T] \}.
\end{displaymath}
It can be verified that $L^2((0,T),K)$ is a closed and convex subset of
$L^2((0,T),V)$. We next state Mosco convergence of function spaces
for parabolic problems.
\begin{definition}
\label{def:moscoPara}
We say that $L^2((0,T),K_n)$ converges to
$L^2((0,T),K)$ in the sense of Mosco if the following conditions hold
\begin{itemize}
 \item[$(M1')$] For every $u \in L^2((0,T),K)$ there exists a sequence 
       $u_n \in L^2((0,T),K_n)$ such that 
       $u_n \rightarrow u$ strongly in $L^2((0,T),V)$.
 \item[$(M2')$] If $(n_k)$ is a sequence of indices converging to $\infty$, 
      $(u_k)$ is a sequence such that $u_k \in L^2((0,T),K_{n_k})$ for 
      every $k$ and 
      $u_k \rightharpoonup  u$ weakly in $L^2((0,T),V)$, then
      $ u \in L^2((0,T),K)$.
\end{itemize}
\end{definition}
It is also useful to define a similar Mosco convergence in $W((0,T),V,V')$ 
when studying domain perturbation for parabolic variational inequalities.
\begin{definition}
\label{def:moscoPara2}
We say that $W((0,T),V,V') \cap L^2((0,T),K_n)$ converges to \linebreak
$W((0,T),V,V') \cap L^2((0,T),K)$ in the sense of Mosco if the following conditions hold
\begin{itemize}
 \item[$(M1")$] for every $u \in W((0,T),V,V') \cap L^2((0,T),K)$ there exists a sequence 
       $u_n \in W((0,T),V,V') \cap L^2((0,T),K_n)$ such that $u_n$ converges strongly to $u$ 
       in $W((0,T),V,V')$.
 \item[$(M2")$] if $(n_k)$ is a sequence of indices converging to $\infty$, 
      $(u_k)$ is a sequence such that $u_k \in W((0,T),V,V') \cap L^2((0,T),K_{n_k})$ for 
      every $k$, 
       $u_k \rightharpoonup u$ weakly in $L^2((0,T),V)$ and $u'_k \rightharpoonup w$ weakly in $L^2((0,T),V')$, then
      $u' = w$ and $u \in W((0,T),V,V') \cap L^2((0,T),K)$.
\end{itemize}
\end{definition}

The following theorem is the key result of this paper.
\begin{theorem}
\label{th:equiMosco}
 The following assertions are equivalent:
 \begin{itemize}
 \item[\upshape(i)] 
  $K_n$ converges to $K$ in the sense of Mosco.
 \item[\upshape(ii)]
  $L^2((0,T),K_n)$ converges to $L^2((0,T),K)$ in the sense of Mosco.
 \item[\upshape(iii)]
  $W((0,T),V,V') \cap L^2((0,T),K_n)$ converges  to
  $W((0,T),V,V') \cap L^2((0,T),K)$ in the sense of Mosco.
 \end{itemize}
\end{theorem}
Before proving the equivalence of Mosco convergences in Theorem
\ref{th:equiMosco}, we require some technical lemmas. 
\begin{lemma}
\label{lem:convInt}
For a bounded open interval $(a,b) \subset \mathbb R$, let $u \in L^2((a,b),K)$. 
If $\phi \in \mathscr D((a,b))$ such that $\int_a^b \phi(t) \;dt
=1$ then $\int_a^b u(t) \phi(t) \;dt \in K$.
\end{lemma}
\begin{proof}
Since $K$ is closed and convex,
$\int_a^b u(t) \phi(t) \;dt \in \overline{\text{\upshape{conv}}}\{u(t) \mid t \in (a,b)\} \subset K$ 
for all $u \in L^2((a,b),K)$.
\end{proof}

\begin{lemma}
\label{lem:mollifierB}
Let $I = (a,b)$ be a bounded open interval in $\mathbb R$.
If $u \in L^2(I,V)$ and $\int_{I} u(t) \phi(t) \;dt \in K$
for all $\phi \in \mathscr D(I)$ with $\int_{I} \phi(t) \;dt =1$, 
then $u \in L^2(J,K)$ for all $J =(c,d) \subset \subset I$.
\end{lemma}
\begin{proof}
Let $\eta \in \mathscr D(\mathbb R)$ be the standard mollifier.
For $\epsilon >0$, we define $\eta_{\epsilon}(t) = \frac{1}{\epsilon}
\eta(\frac{t}{\epsilon})$ so that $\eta_{\epsilon} \in \mathscr D(\mathbb R)$
with $\int_{\mathbb R} \eta_{\epsilon}(t) \;dt = 1$ and $\supp(\eta_{\epsilon}) \subset
(-\epsilon, \epsilon)$.  Consider the mollified function $u_{\epsilon} := \eta_{\epsilon} \ast u$.
For a.e. $t \in I$, we have
\begin{displaymath}
\begin{aligned}
 \|u_{\epsilon}(t) - u(t)\|_V 
 &= \left \| \int_{t-\epsilon}^{t+\epsilon} \eta_{\epsilon}(t-s)[u(s) -u(t)] \;ds \right \|_V \\
 &\leq \frac{1}{\epsilon} \int_{t-\epsilon}^{t+\epsilon} \eta \Big (\frac{t-s}{\epsilon} \Big ) \|u(s)-u(t)\|_V \;ds \\
 &\leq C \frac{1}{\epsilon} \int_{t-\epsilon}^{t+\epsilon} \|u(s)-u(t)\|_V \;ds.\\ 
\end{aligned}
\end{displaymath}
By Lebesgue's differentiation theorem for vector valued functions 
 (Theorem III.12.8 of \cite{MR1009162}), 
$u_{\epsilon}(t) \rightarrow u(t)$ in $V$ a.e. $t \in I$. 
By the definition of $u_{\epsilon}$,
\begin{displaymath}
 u_{\epsilon} (t) = \int_{I} \eta_{\epsilon}(t-s) u(s) \;ds =: \int_{I} u(s) \phi_{\epsilon}(s) \;ds ,
\end{displaymath}
where we set $\phi_{\epsilon}(s) := \eta_{\epsilon}(t-s)$.  
Let $J \subset \subset I$. 
For $t \in J$, we can choose $\epsilon$ sufficiently small so that 
$\supp(\phi_{\epsilon}) = (t-\epsilon, t+\epsilon) \subset I$. 
It follows from the assumption that $u_{\epsilon}(t) \in K$ for all $t \in J$. 
Since $K$ is a closed subset of $V$, the limit point $u(t) \in K$ a.e. $t \in J$.
Hence $u \in L^2(J,K)$ as required.
\end{proof}
\begin{lemma}
\label{lem:densityOfCK}
The set $C([0,T],K)$ is dense in $L^2((0,T),K)$.
\end{lemma}
\begin{proof}
Note first that the lemma is trivial if $K$ is a subspace of $V$ (i.e. $K$
is a Banach space)\cite[Theorem 23.2 (c)]{MR1033497}. 
Let $u \in L^2((0,T),K)$. We choose a function $\phi \in \mathscr D
((0,T))$ with $\int_0^T \phi(t) \;dt = 1$. It follows  from Lemma \ref{lem:convInt} that
$\xi := \int_0^T u(t) \phi(t) \;dt \in K$. Define the extended function
$\tilde u \in L^2((-1,T+1),K)$ by
\begin{displaymath}
 \tilde u (t) := \left \{
                 \begin{aligned}
                  & \xi  &&\quad \text{on } (-1, 0) \cup
                  (T,T+1) \\
                  & u(t) &&\quad \text{on } (0,T) .
                 \end{aligned} 
                 \right .
\end{displaymath}
By a mollification argument, the function
$u_{\epsilon} := \eta_{\epsilon} \ast \tilde u$ belongs to $C(\mathbb R, V)$.
Moreover, $u_{\epsilon}$ converges to $\tilde u$ in $L^2((-1,T+1),V)$.
By Choosing $0< \epsilon < 1$, $u_{\epsilon} (t) \in K$ for
all $t \in [0,T]$. Therefore, the restriction of $u_{\epsilon}$ on
$[0,T]$ belongs to $C([0,T],K)$ and converges to $u$ in $L^2((0,T),V)$ as $\epsilon \rightarrow 0$. 
\end{proof}
\begin{lemma}
\label{lem:densityOfWK}
The set $C^{\infty}([0,T],V) \cap C([0,T],K)$ is dense in $W((0,T),V,V')
\cap L^2((0,T),K)$.
\end{lemma}
\begin{proof}
Let $u \in W((0,T),V,V') \cap L^2((0,T),K)$.
For $\delta > 0$, we define the stretching map $S_{\delta} :[0,T]
\rightarrow [-\delta, T+\delta]$ by
\begin{equation}
\label{eq:stretchMap}
 S_{\delta} (t) := \Big (\frac{T+2\delta}{T} \Big ) t - \delta .
\end{equation}
We define $u_{\delta} \in W((-\delta, T+\delta),V,V') 
\cap L^2((-\delta,T+\delta),K)$ by $u \circ S_{\delta}^{-1}$.
It can be shown that the restriction of $u_{\delta}$ on $(0,T)$
converges to $u$ in $W((0,T),V,V')$ as $\delta \rightarrow 0$. 
Let $\eta_{\epsilon}$ be a mollifier. 
For $t \in [0,T]$ and $\epsilon < \delta$, the
translation of $\eta_{\epsilon}$ by $t$ (denoted by $\eta_{\epsilon, t}$)
belongs to $\mathscr D((-\delta,T+\delta))$. Hence if $\epsilon < \delta$,
$\eta_{\epsilon} \ast u_{\delta}$ belongs to $C^{\infty}([0,T],V) \cap
C([0,T],K)$. Moreover, a mollification argument shows that 
$\eta_{\epsilon} \ast u_{\delta}$ converges to $u_{\delta}$
in $W((0,T),V,V')$ as $\epsilon \rightarrow 0$. The result then follows.
\end{proof}
\begin{proposition}
\label{prop:convCombUdelta}
Suppose Mosco condition $(M1)$ is satisfied.
For $\delta \geq 0$,
let
 $A_{\delta,n} :=\Big \{ \sum_{i=1}^m \phi_i(t) v_i, m \in \mathbb N \Big \}$,
where 
\begin{equation}
\label{eq:convCombAn}
\left \{
\begin{aligned}
    &v_i \in K_n,\phi_i \in C^{\infty}([-\delta,T+\delta]) 
      \quad \text{ for all } i=1, \ldots,m , \\
    &0 \leq \phi_i(t) \leq 1 \quad \text{ for all } t \in [-\delta,T+\delta] 
       \text{ and for all } i=1, \ldots,m, \\ 
    &\sum_{i=1}^m \phi_i(t) = 1 \quad \text{ for all } t \in [-\delta,T+\delta].
\end{aligned}
\right .
\end{equation}
If $ u_{\delta} \in C([-\delta,T+\delta],K)$, then there exists a
sequence of functions $u_{\delta,n} \in A_{\delta,n}$ such that $u_{\delta,n}(t) \rightarrow u_{\delta}(t)$
in $V$ uniformly on $[-\delta, T+\delta]$ as $n \rightarrow \infty$.
\end{proposition}
\begin{proof}
Let $u_{\delta} \in C([-\delta,T+\delta],K)$. We extend $u_{\delta}$ to $\tilde u_{\delta}
\in C(\mathbb R, K)$ by
\begin{displaymath}
\tilde u_{\delta}(t) :=  \left \{
                 \begin{aligned}
                  & u_{\delta}(-\delta)  &&\quad \text{on } (-\infty, -\delta) \\
                  & u_{\delta}(t) &&\quad \text{on } [-\delta,T+\delta] \\
                  & u_{\delta}(T+\delta) &&\quad \text{on } (T+\delta, \infty) .
                 \end{aligned} 
                 \right .
\end{displaymath}
Let $\epsilon > 0$ be arbitrary.
We denote by $B(t):= B_V(\tilde u_{\delta}(t), \epsilon/2)$ the open ball in $V$ about
$\tilde u_{\delta}(t)$ of radius $\epsilon/2$. Let us construct an open covering $\mathscr
O$ of $(-\delta-1, T+\delta+1)$ by
\begin{displaymath}
\mathscr O = \{\tilde u_{\delta}^{-1}(B(t)) \cap
               (-\delta-1, T+\delta+1) \}_{t \in [-\delta,T+\delta]} .
\end{displaymath}
Since $\mathscr O$ is also an open covering of the compact set $[-\delta,T+\delta]$,
there exists a finite subcovering 
\begin{displaymath}
\tilde {\mathscr O} =  \{\tilde u_{\delta}^{-1}(B(t_i)) \cap
               (-\delta-1, T+\delta+1) \}_{i=1,\ldots,m} ,
\end{displaymath}
where $t_i \in [-\delta,T+\delta]$ for all $i=1,\ldots,m$. We can assume that 
$t_1 < t_2 < \ldots < t_m$ and $t_1 = -\delta$, $t_m = T+\delta$ (add them if
required) so that $\tilde {\mathscr O}$ is an open covering of
$[-\delta-1/2, T+\delta+1/2]$.
For each $i \in \{1,\ldots, m\}$, we have $u_{\delta}(t_i) \in K$. Thus, by Mosco
condition (M1), there exists $v_{i,n} \in K_n$ such that
$ \|v_{i,n} - u_{\delta}(t_i) \|_V < \epsilon/2$
if $n > N_i$ for some $N_i \in \mathbb N$. 
Let $N := \max_{i=1,\ldots,m} N_i$. It follows that
$\|v_{i,n} - u_{\delta}(t_i) \|_V < \epsilon/2$
if $n > N$ for all $i \in \{1, \ldots, m \}$. 

Choose a smooth partition of unity
$\{\phi_i\}_{i=1,\ldots m}$ for $[-\delta-1/2, T+\delta+1/2]$ subordinate to
$\tilde {\mathscr O}$. Precisely, we choose $\phi_i$ such that
$\phi_i \in C^{\infty}_0(\tilde  u_{\delta}^{-1}(B(t_i)) \cap
(-\delta-1, T+\delta+1))$  and
$\sum_{i=1}^m \phi_i(t) =1$ for all $t \in [-\delta-1/2, T+\delta+1/2]$.
Define a function $u_{\delta,n}$ on $(-\delta-1,T+\delta+1)$ by
\begin{displaymath}
 u_{\delta,n}(t) := \sum_{i=1}^m \phi_i (t) v_{i,n}. 
\end{displaymath}
It is clear that the restriction of $u_{\delta,n}$ on
$[-\delta,T+\delta]$ belongs to $A_{\delta,n}$ if
$n > N$. Moreover, for $t \in [-\delta,T+\delta]$, 
\begin{displaymath}
\begin{aligned}
 \|u_{\delta,n}(t) - u_{\delta}(t) \|_V 
 &\leq  \sum_{i=1}^m \phi_i(t) \|v_{i,n} - u_{\delta}(t)\|_V  \\
 &\leq  \sum_{i=1}^m \phi_i(t) \|v_{i,n} - u_{\delta}(t_i)\|_V +
       \sum_{i=1}^m \phi_i(t) \|u_{\delta}(t_i) - u_{\delta}(t)\|_V  \\
 &<  \epsilon/2 + \epsilon/2 = \epsilon, 
\end{aligned} 
\end{displaymath}
if $n > N$.
Note that $m$ and $N$ chosen above depend on $\epsilon$. As the above argument holds for
each fixed $\epsilon$, we conclude that for every $\epsilon > 0$, there exists
a sequence $u_{\delta,n}^{\epsilon} \in A_{\delta,n}$ and $N(\epsilon) \in \mathbb N$ such that
\begin{displaymath}
  \| u_{\delta,n}^{\epsilon}(t) - u_{\delta}(t)\|_V \leq \epsilon,
\end{displaymath}
for all $t \in [-\delta,T+\delta]$ if $n > N(\epsilon)$.

In particular, for every $k \in \mathbb N$ we can find
a sequence 
$u_{\delta,n}^{k} \in A_{\delta,n}$ and $N_k \in \mathbb N$ such that
\begin{equation}
\label{eq:diagonalRangeUdelta}
  \| u_{\delta,n}^{k}(t) - u_{\delta}(t)\|_V \leq \frac{1}{k},
\end{equation}
for all $t \in [-\delta,T+\delta]$ if $n > N_k$. By choosing inductively 
we can assume that $N_k < N_{k+1}$ for all $k \in \mathbb N$. We extract a
sequence of the form
\begin{displaymath}
  u_{\delta,1}^1, u_{\delta,2}^1, \ldots, u_{\delta,(N_1+1)}^1, \ldots, u_{\delta, N_2}^1,
    u_{\delta,(N_2+1)}^2, \ldots, u_{\delta,N_3}^2,
    u_{\delta,(N_3+1)}^3, \ldots, u_{\delta, N_4}^3,
    \ldots 
\end{displaymath}
so that the $n$-th element of this sequence belongs to $A_{\delta,n}$ for all
$n \in \mathbb N$. Moreover, by \eqref{eq:diagonalRangeUdelta}, we see that
this sequence converges to $u_{\delta}$  
uniformly with respect to $t \in [-\delta,T+\delta]$ as $n \rightarrow \infty$.
This proves the statement of the proposition.   
\end{proof}

We are now in a position to prove our main result.
\begin{proof}[Proof of Theorem \ref{th:equiMosco}] The proof is divided into four parts
including $(i) \Rightarrow (ii)$, $(ii) \Rightarrow (i)$, $(i) \Rightarrow (iii)$ and
$(iii) \Rightarrow (i)$. For $(i) \Rightarrow (ii)$, we actually show that 
$(M1) \Rightarrow (M1')$ and $(M2) \Rightarrow (M2')$. The other three directions are 
proved in the same way.

$(i) \Rightarrow(ii)$:
Let $u \in L^2((0,T),K)$. By the density of $C([0,T],K)$ in $L^2((0,T),K)$
(Lemma \ref{lem:densityOfCK}), we may assume that $u \in C([0,T],K)$. 
We apply Proposition \ref{prop:convCombUdelta} with $\delta = 0$ to obtain 
a sequence of functions $u_n \in L^2((0,T),K_n)$ such that 
$u_n(t) \rightarrow u(t)$ in $V$ uniformly on $[0,T]$. The uniform convergence 
on $[0,T]$ implies that $u_n \rightarrow u$ in $L^2((0,T),V)$, showing $(M1')$. 
To prove condition $(M2')$, suppose $(n_k)$ is a sequence of indices
converging to $\infty$, $(u_k)$ is a sequence such that $u_k \in
L^2((0,T),K_{n_k})$ for every $k$ and $u_k \rightharpoonup u$ in
$L^2(0,T),V)$. By the definition of weak convergence, 
\begin{equation}
\label{eq:weakCovMoscoiToii}
 \int_0^T \langle w(t), u_k(t) \rangle \;dt \rightarrow 
 \int_0^T \langle w(t), u(t) \rangle \;dt,
\end{equation}
for all $w \in L^2((0,T),V')$. 
By taking $w$ of the form $w = \xi \phi(t)$ where $\xi \in V'$  and $\phi \in \mathscr
D((0,T))$ in \eqref{eq:weakCovMoscoiToii} and applying a basic property
of Bochner-Lebesgue space \cite[Proposition 23.9(a)]{MR1033497}, it
follows that 
\begin{equation}
\label{eq:weakConvIntUkPhi}
\int_0^T u_k(t) \phi(t) \;dt \rightharpoonup 
\int_0^T u(t) \phi(t) \;dt
\end{equation}
weakly in $V$ for all $\phi \in \mathscr D((0,T))$.
Let $\phi_0 \in \mathscr D((0,T))$ with
$\int_0^T \phi_0(t) \;dt = 1$ and define 
$\zeta_k := \int_0^T u_k(t) \phi_0(t) \;dt$. Lemma \ref{lem:convInt}
implies that $\zeta_k \in K_{n_k}$ for all $k \in \mathbb N$. 
Since $\zeta_k \rightharpoonup \zeta := \int_0^T u(t) \phi_0(t) \;dt$ by
\eqref{eq:weakConvIntUkPhi}, Mosco
condition $(M2)$ implies that $\zeta \in K$.
We now extend
$u_k$ to $\tilde u_k \in L^2((-1,T+1),K_{n_k})$ by
\begin{equation}
\label{eq:extWeakMosco}
 \tilde u_k (t) := \left \{
                 \begin{aligned}
                  & \zeta_k  &&\quad \text{on } (-1, 0) \cup
                  (T,T+1) \\
                  & u_k(t) &&\quad \text{on } (0,T) .
                 \end{aligned} 
                 \right .
\end{equation}
It can be easily seen that $\tilde u_k \rightharpoonup \tilde u$ weakly in
$L^2((-1,T+1),V)$ , where 
$\tilde u$ defined as \eqref{eq:extWeakMosco} with $k$ deleted.
Using the definition of weak convergence in $L^2((-1,T+1),V)$ and
a similar argument as above,  we obtain
$\int_{-1}^{T+1} \tilde u_k(t) \phi(t) \;dt \rightharpoonup 
\int_{-1}^{T+1} \tilde u(t) \phi(t) \;dt$ weakly in $V$ 
for all $\phi \in \mathscr D((-1,T+1))$. In particular, taking $\phi \in
\mathscr D((-1,T+1))$ with $\int_{-1}^{T+1} \phi(t) \;dt =1$, we have 
$\int_{-1}^{T+1} \tilde u_k(t) \phi(t) \;dt \in K_{n_k}$ converges weakly
to $\int_{-1}^{T+1} \tilde u(t) \phi(t) \;dt$ in $V$. Thus, Mosco
condition $(M2)$ implies $\int_{-1}^{T+1} \tilde u(t) \phi(t) \;dt \in K$
for all $\phi \in \mathscr D((-1,T+1))$ with $\int_{-1}^{T+1} \phi(t) \;dt
=1$. By Lemma \ref{lem:mollifierB}, we conclude that $u \in L^2((0,T),K)$ and Mosco
condition $(M2')$ follows.

$(ii) \Rightarrow  (i)$:
Let $u \in K$. Define $v \in L^2((0,T),K)$ by the constant function
$v(t) :=u$ for $t \in (0,T)$. By condition $(M1')$, there 
exists $(v_n)_{n \in \mathbb N}$ with $v_n \in L^2((0,T),K_n)$ such that
$v_n \rightarrow v$ in $L^2((0,T),V)$. Let $\phi_0 \in \mathscr
D((0,T))$ with $\int_0^T \phi_0(t) \;dt = 1$. We show that the sequence 
$(u_n)_{n \in \mathbb N}$ defined by $u_n := \int_0^T v_n(t) \phi_0(t) \;dt$ 
gives Mosco condition $(M1)$. First note that $u_n \in K_n$ for 
all $n \in \mathbb N$ by Lemma \ref{lem:convInt}. Moreover,
\begin{displaymath}
 \begin{aligned}
  \| u_n - u \|_V &= \Big \| \int_0^T v_n(t) \phi_0(t) \;dt - u \Big \|_V \\
  &= \Big \| \int_0^T [v_n(t) \phi_0(t) - v(t) \phi_0(t)] \;dt \Big \|_V \\
  &\leq \int_0^T |\phi_0(t)|\|v_n(t) -v(t)\|_V \;dt \\
  &\leq \sqrt T \Big ( \int_0^T \|v_n(t) - v(t)\|_V^2 \;dt
                 \Big )^{\frac{1}{2}} \|\phi_0\|_{\infty}  \\
  &\rightarrow 0, 
 \end{aligned}
\end{displaymath}
as $n \rightarrow \infty$. 
To prove condition $(M2)$, suppose $(n_k)$ is a sequence of indices
converging to $\infty$, $(u_k)$ is a sequence such that $u_k \in K_{n_k}$ 
for every $k$ and $u_k \rightharpoonup u$ in $V$. Define $v_k \in L^2((0,T),K_{n_k})$
by the constant function $v_k(t) := u_k$ for $t \in (0,T)$. 
It can be easily verified that $v_k \rightharpoonup v$ in
$L^2((0,T),V)$, where $v$ is the constant function $v(t):= u$  
for $t \in (0,T)$.
It follows from Mosco condition $(M2')$ that $v \in L^2((0,T),K)$.
Hence $u \in K$ as required.

$(i) \Rightarrow (iii)$:
Let $u \in W((0,T),V,V') \cap L^2((0,T),K)$. By Lemma \ref{lem:densityOfWK}, 
we may assume that $u \in C^{\infty}([0,T],V) \cap C([0,T],K)$.
For $\delta > 0$, 
we define the stretched function 
$u_{\delta} \in C^{\infty}([-\delta,T+\delta],V) \cap
C([-\delta,T+\delta],K)$ by $u_{\delta} = u \circ S_{\delta}^{-1}$, 
where $S_{\delta}$ is the stretching map given by \eqref{eq:stretchMap}. 
It can be shown that the restriction of $u_{\delta}$ on $[0,T]$
converges to $u$ in $W((0,T),V,V')$ as $\delta \rightarrow 0$. 
By Proposition \ref{prop:convCombUdelta}, there exists a sequence of functions
$u_{\delta,n} \in A_{\delta,n}$ such that $u_{\delta,n} (t) \rightarrow
u_{\delta}(t)$ uniformly on $[-\delta, T+\delta]$ as $n \rightarrow
\infty$. Let $\eta_{1/j}$ be a mollifier.  
For $t \in [0,T]$ and $j > 1/{\delta}$, the
translation of $\eta_{1/j}$ by $t$ (denoted by $\eta_{1/j, t}$)
belongs to $\mathscr D((-\delta,T+\delta))$. Hence if $j > 1/{\delta}$, we have
$\eta_{1/j} \ast u_{\delta,n} \in C^{\infty}([0,T],V) \cap
C([0,T],K_n)$. 
By continuity of convolution and the well known fact on the $r$-th
order derivative that
\begin{displaymath}
 \frac{d^r}{dt^r}(\eta_{1/j} \ast u_{\delta,n})
  =  \frac{d^r}{dt^r} \eta_{1/j} \ast u_{\delta,n} 
  = \eta_{1/j} \ast \frac{d^r}{dt} u_{\delta,n},
\end{displaymath}
we deduce that
$\eta_{1/j} \ast u_{\delta,n} \rightarrow \eta_{1/j}
\ast u_{\delta}$ in $C^{\infty}([0,T],V)$ as $n \rightarrow \infty$. 
Similarly, $\eta_{1/j} \ast u_{\delta} \rightarrow u_{\delta}$ in 
$C^{\infty}([0,T],V)$ as $j \rightarrow \infty$. 
The above shows that we can construct a function of the form
$\eta_{1/j} \ast u_{\delta,n} \in W((0,T),V,V') \cap L^2((0,T),K_n)$ converging to
$u$ in $W((0,T),V,V')$. Hence Mosco condition $(M1")$ follows.
To prove condition $(M2")$, suppose $(n_k)$ is a sequence of indices
converging to $\infty$, $(u_k)$ is a sequence such that $u_k \in
W((0,T),V,V') \cap L^2((0,T),K_{n_k})$ for every $k$, $u_k \rightharpoonup u$ in
$L^2((0,T),V)$ and $u'_k \rightharpoonup  w$ in $L^2((0,T),V')$. 
Since $V$ is continuously embedded in $V'$, it follows immediately 
that $u'=w$ and hence $u \in W((0,T),V,V')$ (see \cite[Proposition 23.19]{MR1033497}).
Using $(i) \Rightarrow (ii)$, specifically Mosco condition $(M2')$, 
we conclude that $u \in W((0,T),V,V') \cap L^2((0,T),K)$.

$(iii) \Rightarrow (i)$: 
Let $u \in K$. Define $v \in W((0,T),V,V') \cap L^2((0,T),K)$ by the constant function
$v(t) :=u$ for $t \in (0,T)$. By condition $(M1")$, there 
exists $(v_n)_{n \in \mathbb N}$ with $v_n \in W((0,T),V,V') \cap L^2((0,T),K_n)$ such that
$v_n \rightarrow v$ in $W((0,T),V,V')$. In particular, $v_n$ converges strongly to $v$ in 
$L^2((0,T),V)$. By the same argument as in the proof of $(ii) \Rightarrow (i)$, 
we can show that $u_n := \int_0^T v_n(t) \phi_0(t) \;dt$, for some $\phi_0 \in \mathscr D((0,T))$ 
with $\int_0^T \phi_0(t) \;dt =1$ establishes Mosco condition $(M1)$.
To prove condition $(M2)$, suppose $(n_k)$ is a sequence of indices
converging to $\infty$, $(u_k)$ is a sequence such that $u_k \in K_{n_k}$ 
for every $k$ and $u_k \rightharpoonup u$ in $V$. 
Define $v_k \in W((0,T),V,V') \cap L^2((0,T),K_{n_k})$
by the constant function $v_k(t) := u_k$ for $t \in (0,T)$. 
By the same argument as in the proof of $(ii) \Rightarrow (i)$, we have 
$v_k \rightharpoonup v$ in $L^2((0,T),V)$, where $v(t):= u$  
for $t \in (0,T)$. Moreover, it is clear that $v'_k = 0$ for all 
$k \in \mathbb N$ and hence $v'_k \rightharpoonup v'=0$ in $L^2((0,T),V')$.
We apply $(M2")$ to deduce that $v \in W((0,T),V,V') \cap L^2((0,T),K)$.
Hence $u \in K$.
\end{proof}

\section{Application in domain perturbation for parabolic equations}
\label{sec:stabilityParaEq}

In this section, we study the behaviour of solutions of parabolic equations subject to
Dirichlet boundary condition and Neumann boundary condition under 
domain perturbation.
Let $\Omega_n, \Omega$ be bounded open
sets in $\mathbb R^N$ and $D \subset \mathbb R^N$ be a ball such that
$\Omega_n , \Omega \subset D$ for all $n \in \mathbb N$.
Suppose $a_{ij}, a_i, b_i, c_0$ are functions in $L^{\infty}(D \times (0,T))$ 
and $a_{ij}$ satisfies ellipticity condition. 
More precisely, there exists $\alpha > 0$ such that 
$a_{ij}(x,t) \xi_i \xi_j \geq \alpha |\xi|^2$
for all $\xi \in \mathbb R^N$. We consider the evolution triple 
$V_n \overset{d}{\hookrightarrow} H_n \overset{d}{\hookrightarrow} V_n'$, 
where we choose 
\begin{itemize}
 \item $V_n = H^1_0(\Omega_n)$ and $H_n = L^2(\Omega_n)$ for Dirichlet problem
 \item $V_n = H^1(\Omega_n)$ and $H_n = L^2(\Omega_n)$ for Neumann problem.
\end{itemize}
For $t \in (0,T)$, suppose $a_n(t;\cdot, \cdot)$ is a bilinear form on $V_n$ defined by
\begin{equation}
 a_n(t;u,v) := \int_{\Omega_n} [a_{ij}(x,t) \partial_j u + a_i(x,t)u] \partial_i v
                + b_i(x,t) \partial_i u v + c_0(x,t) uv \;dx.
\end{equation}
It follows that for all $n \in \mathbb N$, there exist three constants
$M >0, \alpha >0$ and $\lambda \in \mathbb R$ independent of $t \in [0,T]$ 
such that
\begin{equation}
 \label{eq:biFormContOnN}
    |a_n(t;u,v)| \leq M \|u\|_{V_n} \|v\|_{V_n} ,
\end{equation}
for all $u,v \in V_n$ and
\begin{equation}
 \label{eq:biFormCoerOnN}
    a_n(t;u,u) + \lambda \|u\|^2_{H_n} \geq \alpha \|u\|^2_{V_n} ,
\end{equation}
for all $u \in V_n$.
Given $u_{0,n} \in L^2(D)$ and $f_n \in L^2(D\times (0,T))$, 
let us consider the following boundary value problem in $\Omega_n \times (0,T]$.
\begin{equation}
\label{eq:paraOmegaN}
  \left \{ 
  \begin{aligned}
    \frac{\partial u}{\partial t} + \mathcal A_n(t) u &= f_n(x,t) &&\quad \text { in } \Omega_n \times (0,T]\\
    \mathcal B_n(t) u &= 0  &&\quad \text{ on } \partial \Omega_n \times (0,T]\\
     u(\cdot,0) &= u_{0,n}  &&\quad \text{ in } \Omega_n , \\
  \end{aligned} 
  \right . 
\end{equation}
where $\mathcal A_n$ and $\mathcal B_n$ are operators on $V_n$ given by
\begin{displaymath}
 \mathcal A_n(t) u := - \partial_i[a_{ij}(x,t) \partial_j u + a_i(x,t)u] + b_i(x,t) \partial_i u 
                       + c_0(x,t) u ,
\end{displaymath}
and $\mathcal B_n$ is one of the following 
\begin{displaymath}
\begin{aligned}
 \mathcal B_n(t) u &:= u &&\quad \text{Dirichlet boundary condition}  \\
 \mathcal B_n(t) u &:= [a_{ij}(x,t) \partial_j u + a_i(x,t)u] \;\nu_i
  &&\quad \text{Neumann boundary condition.} \\
\end{aligned}
\end{displaymath}
We wish to show that a sequence of solutions of
the above parabolic equations in $\Omega_n \times (0,T]$ converges to the solution of the
following limit problem
\begin{equation}
\label{eq:paraOmega}
  \left\{ 
  \begin{aligned}
    \frac{\partial u}{\partial t} + \mathcal A(t) u &= f(x,t) &&\quad \text { in } \Omega \times (0,T]\\
    \mathcal B(t) u &= 0  &&\quad \text{ on } \partial \Omega \times (0,T]\\
     u(\cdot,0) &= u_0 &&\quad \text{ in } \Omega. \\
  \end{aligned} 
  \right . 
\end{equation}
However, we will consider the boundary value problems \eqref{eq:paraOmegaN} and
\eqref{eq:paraOmega} in the abstract form. As discussed in Section \ref{sec:prelim}, 
we can write \eqref{eq:paraOmegaN} as
\begin{equation}
\label{eq:paraAbsOmegaN}
  \left \{
  \begin{aligned}
   u'(t) + A_n(t) u &= f_n(t) \quad \text{ for } t \in (0,T] \\
   u(0) &= u_{0,n},\\
  \end{aligned}
  \right .
\end{equation}
where $A_n(t) \in \mathscr L(V_n , V_n')$ is the operator
induced by the bilinear form $a_n(t;,\cdot,\cdot)$. Similarly, we write
\eqref{eq:paraOmega} as
\begin{equation}
\label{eq:paraAbsOmega}
  \left \{
  \begin{aligned}
   u'(t) + A(t) u &= f(t) \quad \text{ for } t \in (0,T] \\
   u(0) &= u_0.\\
  \end{aligned}
  \right .
\end{equation}
Throughout this section, we denote the variational
solution of \eqref{eq:paraAbsOmegaN} by $u_n$ and 
the variational solution of \eqref{eq:paraAbsOmega} by $u$.
We illustrate an application of Mosco convergence 
to obtain stability of variational solutions under domain perturbation.
The proof is motivated by the techniques presented in \cite{MR1404388}. 
However, we replace the notion of convergence of domains ((3.5) and (3.6) in \cite{MR1404388}) by
Mosco convergence. It is not difficult to see that the assumption on domains in \cite{MR1404388}
implies that $H^1_0(\Omega_n)$ converges to $H^1_0(\Omega)$ in the sense of Mosco.

\subsection{Dirichlet problems}
\label{subsec:dirProb}
 
When the domain is perturbed, the variational solutions belong to different function spaces.
We often extend functions by zero outside the domain.
We embed the spaces $H^1_0(\Omega_n)$ into $H^1(D)$ by $v \mapsto \tilde v$,
where $\tilde v = v$ on $\Omega_n$ and $\tilde v = 0$ on $D \backslash \Omega_n$.
Similarly, we may consider the embedding $L^2((0,T),H^1_0(\Omega_n))$ into $L^2((0,T),H^1(D))$
by $w(t) \mapsto \tilde w(t)$ for a.e. $t \in (0,T)$. 
Note that the trivial extension $\tilde v$ also acts on $L^2(\Omega_n)$ into $L^2(D)$.

Let us take $V := H^1(D)$,  $K_n := H^1_0(\Omega_n)$  and $K := H^1_0(\Omega)$, 
and consider Mosco convergence of $K_n$ to $K$. In this case $K_n$
and $K$ are closed and convex subsets of $V$ in the sense of the above embedding. 
In fact, $K_n$ and $K$ are closed subspace of $V$. The main application 
of Theorem \ref{th:equiMosco} is to show that the variational
solution $u_n$ of \eqref{eq:paraAbsOmegaN} converges to the variational solution $u$ of \eqref{eq:paraAbsOmega} by applying 
various Mosco conditions.

\begin{theorem}
 \label{th:weakConvSolDir}
  Suppose $\tilde f_n \rightarrow \tilde f$ in $L^2((0,T),L^2(D))$ and $\tilde u_{0,n} \rightarrow \tilde u_0$ in $L^2(D)$.
  If $H^1_0(\Omega_n)$ converges to $H^1_0(\Omega)$ in the sense of Mosco, then
  $\tilde u_n$ converges weakly to $\tilde u$ in $L^2((0,T),H^1(D))$.
\end{theorem}
\begin{proof}
  Since $f_n \rightarrow f$ in $L^2((0,T),L^2(D))$ and $u_{0,n} \rightarrow u_0$ in $L^2(D)$,
 it follows from \eqref{eq:unifBoundSol} that
 $\|u_n \|_{W(0,T,V_n,V_n')}$ is uniformly bounded. Hence 
 $\tilde u_n$ is uniformly bounded in $L^2((0,T),H^1(D))$.
 We can extract a subsequence (denoted again by $u_n$), such that
 $\tilde u_n \rightharpoonup w$ in $L^2((0,T),H^1(D))$.
 Mosco condition $(M2')$ (from Theorem \ref{th:equiMosco}) implies that 
 $w \in L^2((0,T),H^1_0(\Omega))$. 
 It remains to show that  $w=u$ in $L^2((0,T),H^1_0(\Omega))$.
 
 Let $\xi \in H^1_0(\Omega)$ and $\phi \in \mathscr D([0,T))$. 
 Mosco condition ($M1)$ implies that there exists $\xi_n \in H^1_0(\Omega_n)$ such that
 $\tilde \xi_n \rightarrow \tilde \xi$ in $H^1(D)$.
 As $u_n$ is the variational solution of \eqref{eq:paraAbsOmegaN}, we get from \eqref{eq:solFormula} that
  \begin{displaymath}
   \begin{aligned}
    - \int_0^T (u_n(t)|\xi_n) \phi'(t) \;dt + \int_0^T a_n(t; u_n(t), \xi_n ) \phi(t) \;dt \\
    = (u_{0,n}|\xi_n) \phi(0) + \int_0^T \langle f_n(t), \xi_n \rangle \phi(t) \;dt .
   \end{aligned}
  \end{displaymath}
   By letting $n \rightarrow \infty$, we get
   \begin{equation}
   \label{eq:limitIsWsolDir}
   \begin{aligned}
    - \int_0^T (w(t)|\xi) \phi'(t) \;dt + \int_0^T a(t; w(t), \xi ) \phi(t) \;dt \\
    = (u_{0}|\xi) \phi(0) + \int_0^T \langle f(t), \xi \rangle \phi(t) \;dt .
   \end{aligned}
  \end{equation}
Hence $w$ is a variational solution of \eqref{eq:paraAbsOmega}. By
the uniqueness of solution, we conclude that $w=u$ in $L^2((0,T),H^1_0(\Omega))$ and the whole 
sequence converges. 
\end{proof}
\begin{lemma}
\label{lem:L2weakLimUnDir}
Suppose $\tilde f_n \rightarrow \tilde f$ in $L^2((0,T),L^2(D))$ and
$\tilde u_{0,n}\rightarrow \tilde u_0$ in $L^2(D)$. If $H^1_0(\Omega_n)$
converges to $H^1_0(\Omega)$ in the sense of Mosco, then
for each $t \in [0,T]$ we have $\tilde u_n(t)_{|_{\Omega}} \rightharpoonup  u(t)$ in $L^2(\Omega)$.
\end{lemma} 
\begin{proof}
Since $\tilde f_n$ is uniformly bounded in $L^2((0,T),L^2(D))$ and 
  $\tilde u_{0,n}$ is uniformly bounded in $L^2(D)$, we have from \eqref{eq:boundOfSol} that
\begin{displaymath}
 \max_{t \in [0,T]} \|\tilde u_n(t)\|_{L^2(D)} \leq M,
\end{displaymath}
for some $M > 0$. Hence for a subsequence denoted again by $u_n(t)$,
there exists $w \in L^2(\Omega)$ such that $\tilde u_n(t)_{|_{\Omega}} \rightharpoonup  w$ in $L^2(\Omega)$.
Let $\xi \in H^1_0(\Omega)$ and $\phi \in \mathscr D((0,t])$.
 Mosco condition $(M1)$ implies that there exists $\xi_n \in H^1_0(\Omega_n)$ such that 
$\tilde \xi_n \rightarrow \tilde \xi$ in $H^1(D)$. 
As $u_n$ is the variational solution of \eqref{eq:paraAbsOmegaN}, we have
\begin{displaymath}
\begin{aligned}
 - \int_0^t (u_n(s)|\xi_n) \phi'(s) \;ds + \int_0^t a_n(s;u_n(s),\xi_n) \phi(s) \;ds \\
= -(\tilde u_n(t)| \tilde \xi_n)_{L^2(D)} \phi(t) + \int_0^t  \langle f_n(s), \xi_n \rangle \phi(s) \;ds.
 \end{aligned}
\end{displaymath}
Now
\begin{displaymath}
 (\tilde u_{0,n})| \tilde \xi_n)_{L^2(D)} = ( \tilde u_{0,n} | \tilde \xi_n)_{L^2(\Omega)}+
 (\tilde u_{0,n}| \tilde \xi_n)_{L^2(D \backslash \Omega)}.
\end{displaymath}
 Since $\tilde \xi_n \rightarrow \tilde \xi$ in $L^2(D)$, we have $\tilde \xi_n |_{\Omega} \rightarrow  \xi$ 
 in $L^2(\Omega)$ and $ \tilde \xi_n |_{(D \backslash \Omega)} \rightarrow 0$ in
 $L^2(D \backslash \Omega)$.  Applying the dominated convergence theorem in the second term above
 and using the weak convergence of initial condition $u_{0,n}$ in the first term above, we
 see that 
 \begin{displaymath}
  (u_{0,n}| \xi_n)_{L^2(\Omega_n)} \rightarrow ( u_0 |\xi)_{L^2(\Omega)}.
 \end{displaymath}
%
%
Hence, 
\begin{equation}
\label{eq:weakL2}
 \begin{aligned}
 - \int_0^t (u(s)|\xi) \phi'(s) \;ds + \int_0^t a(s;u(s),\xi) \phi(s) \;ds \\
= -(w|\xi)_{L^2(\Omega)} \phi(t) + \int_0^t  \langle f(s), \xi \rangle \phi(s) \;ds,
 \end{aligned}
\end{equation}
as $n \rightarrow \infty$. As $u$ is the variational solution of \eqref{eq:paraAbsOmega}, 
a similar equation holds with $(w| \xi)_{L^2(\Omega)}$ 
replaced by $(u(t)|\xi)_{L^2(\Omega)}$. 
Therefore $(w| \xi)_{L^2(\Omega)}=(u(t)|\xi)_{L^2(\Omega)}$ for all
$\xi \in H^1_0(\Omega)$.
By the density of $H^1_0(\Omega)$ in $L^2(\Omega)$, $w =u(t)$. 
Hence, for subsequences $\tilde u_n(t)_{|_{\Omega}} \rightharpoonup u(t)$ in $L^2(\Omega)$.
By the uniqueness, the whole sequence $\tilde u_n(t)_{|_{\Omega}}$
converges weakly to $u(t)$ in $L^2(\Omega)$.
\end{proof} 
\begin{remark}
\label{rem:dirWeakInitialF}
In fact, we only require that $\tilde f_n |_{\Omega} \rightharpoonup f$ weakly in
$L^2((0,T),L^2(\Omega))$ and $\tilde u_{0,n}|_{\Omega} \rightharpoonup  u_0$ weakly in $L^2(\Omega)$ to obtain
the conclusion of Theorem \ref{th:weakConvSolDir} and Lemma \ref{lem:L2weakLimUnDir}
as done in \cite{MR1404388}. 
\end{remark}
Next we show the strong convergence of solutions. The assumptions on strong 
convergence of initial values $u_{0,n}$ and inhomogeneous data $f_n$ are 
required in the proof below (see also Remark \ref{rem:NoNeedStrongConvDir} below).
\begin{theorem}
\label{th:strongConvSolDir}
Suppose $\tilde f_n \rightarrow \tilde f$ in $L^2((0,T),L^2(D))$ and
$\tilde u_{0,n}\rightarrow \tilde u_0$ in $L^2(D)$. If $H^1_0(\Omega_n)$
converges to $H^1_0(\Omega)$ in the sense of Mosco, then
$\tilde u_n$ converges strongly to $\tilde u$ in $L^2((0,T),H^1(D))$.
\end{theorem}
\begin{proof}
We have $\tilde u_n$ converges weakly to $\tilde u$ 
in $L^2((0,T),H^1(D))$ from Theorem \ref{th:weakConvSolDir}.
Mosco condition $(M1')$ (from Theorem \ref{th:equiMosco}) implies that
there exists $w_n \in L^2((0,T),H^1_0(\Omega_n))$
such that $\tilde w_n \rightarrow \tilde u$ in $L^2((0,T),H^1(D))$.
For $t \in [0,T]$, we consider 
\begin{equation}
\label{eq:dNDir}
 \begin{aligned}
  d_n(t) &= \frac{1}{2} \|\tilde u_n(t) - \tilde u(t) \|^2_{L^2(D)}
          + \alpha \int_0^t \|\tilde u_n(s) - \tilde w_n(s) \|^2_{H^1(D)} ds.
 \end{aligned}
\end{equation}
By \eqref{eq:biFormCoerOnN} (with $\lambda =0$), we have
\begin{equation}
\label{eq:dNleqDir}
 \begin{aligned}
 d_n(t) &\leq \frac{1}{2} \|\tilde u_n(t)\|^2_{L^2(D)} + \int_0^t a_n(s;u_n(s),u_n(s)) \;ds \\
        & \quad + \frac{1}{2} \|\tilde u(t)\|^2_{L^2(D)} + \int_0^t a_n(s; w_n(s), w_n(s)) \;ds \\
        & \quad -(\tilde u_n(t)| \tilde u(t))_{L^2(D)} - \int_0^t a_n(s; u_n(s),w_n(s)) \;ds \\
        & \quad - \int_0^t a_n(s; w_n(s), u_n(s)) \;ds,\\
 \end{aligned}
\end{equation}
for all $n \in \mathbb N$.
It can be easily seen from the weak convergence of $\tilde u_n$ and the strong
convergence of $\tilde w_n$ to $\tilde u$ in $L^2((0,T),H^1(D))$ that
\begin{equation}
\label{eq:wUnsWnDir}
\begin{aligned}
 &\lim_{n \rightarrow \infty} 
      \Big [\int_0^t a_n(s;u_n(s),w_n(s)) \;ds + \int_0^t a_n(s;w_n(s),u_n(s)) \;ds \Big ] \\
  & \quad = 2 \int_0^t a(s;u(s),u(s)) \;ds ,
\end{aligned}
\end{equation}
and
\begin{equation}
\label{eq:sWnWnDir}
 \lim_{n \rightarrow \infty} \int_0^t a_n(s;w_n(s),w_n(s)) \;ds
   = \int_0^t a(s;u(s),u(s)) \;ds.
\end{equation}
Also, by lemma \ref{lem:L2weakLimUnDir}, we have
\begin{equation}
\label{eq:limUnNormDir}
\lim_{n \rightarrow \infty} (\tilde u_n(t) | \tilde u(t))_{L^2(D)}
 = \lim_{n \rightarrow \infty} (\tilde u_n(t)_{|_{\Omega}} | u(t))_{L^2(\Omega)}
 = \| u(t)\|^2_{L^2(\Omega)}.
\end{equation}
Finally, as $u_n$ is the variational solution of \eqref{eq:paraAbsOmegaN} we get from
\eqref{eq:intByParts} that
\begin{displaymath}
 \begin{aligned}
 \frac{1}{2} \|u_n(t)\|^2_{L^2(\Omega_n)}+ \int_0^t a_n(s; u_n(s),u_n(s)) \;ds \\
 =\frac{1}{2} \|u_n(0)\|^2_{L^2(\Omega_n)} + \int_0^t \langle f_n(s), u_n(s) \rangle \;ds.
 \end{aligned}
\end{displaymath}
By the assumption that $\tilde u_{0,n} \rightarrow \tilde u_0$ strongly in $L^2(D)$ 
and $\tilde f_n \rightarrow \tilde f$ strongly in $L^2((0,T),L^2(D))$, we get
\begin{equation}
\label{eq:limUnPlusAnDir}
 \begin{aligned}
& \lim_{n \rightarrow \infty} \Big [ \frac{1}{2} \|u_n(t)\|^2_{L^2(\Omega_n)}+ \int_0^t a_n(s; u_n(s),u_n(s)) \;ds \Big ]\\
& \quad  =\frac{1}{2} \|u(0)\|^2_{L^2(\Omega)} + \int_0^t \langle f(s), u(s) \rangle \;ds \\
& \quad =\frac{1}{2} \|u(t)\|^2_{L^2(\Omega)}+ \int_0^t a(s; u(s),u(s)) \;ds.
 \end{aligned}
\end{equation}
Hence, it follows from \eqref{eq:dNleqDir} -- \eqref{eq:limUnPlusAnDir} that 
$d_n(t) \rightarrow 0$ for all $t \in [0,T]$. This shows pointwise
convergence of $\tilde u_n(t)$ to $\tilde u(t)$ in $L^2(D)$.
Moreover, by taking $t=T$ we get
\begin{displaymath}
\begin{aligned}
 &\int_0^T \|\tilde u_n(s)-\tilde u(s)\|^2_{H^1(D)} \;ds  \\
 & \quad \leq \int_0^T \|\tilde u_n(s)-\tilde w_n(s)\|^2_{H^1(D)} \;ds
  + \int_0^T \|\tilde w_n(s)-\tilde u(s)\|^2_{H^1(D)} \;ds \\
 & \quad \rightarrow 0, 
\end{aligned}
\end{displaymath}
as $n \rightarrow \infty$. This proves the strong convergence $\tilde u_n \rightarrow \tilde u$
in $L^2((0,T),H^1(D))$.
\end{proof}
In the next theorem we prove convergence of solutions in a stronger norm.
We show that Mosco convergence is suf\mbox{}ficient for uniform convergence of 
solutions in $L^2(D)$ with respect to $t \in [0,T]$. 
We require the following result
on Mosco convergence of $H^1_0(\Omega_n)$ to $H^1_0(\Omega)$
 \cite[Proposition 6.3]{MR1955096}.
\begin{lemma}
\label{lem:M1capH10}
The following statements are equivalent.
\begin{enumerate}
\item  Mosco condition $(M1)$: for every $w \in H^1_0(\Omega)$, there
  exists a sequence $w_n \in H^1_0(\Omega_n)$ such that
  $\tilde w_n \rightarrow \tilde w$ in $H^1(D)$.
\item  $\text{\upshape{cap}} (K \cap \Omega_n^c) \rightarrow 0$ as $n \rightarrow \infty$
   for all compact set $K \subset \Omega$.
\end{enumerate}
\end{lemma}
\begin{theorem}
\label{th:unifConvSolDir}
Suppose $\tilde f_n \rightarrow \tilde f$ in $L^2((0,T),L^2(D))$ and
$\tilde u_{0,n}\rightarrow \tilde u_0$ in $L^2(D)$. If $H^1_0(\Omega_n)$
converges to $H^1_0(\Omega)$ in the sense of Mosco, then
$\tilde u_n$ converges strongly to $\tilde u$ in $C([0,T],L^2(D))$.
\end{theorem} 
\begin{proof}
We notice from the proof of Theorem \ref{th:strongConvSolDir} that 
\begin{displaymath}
 \int_0^t \|\tilde u_n(s) - \tilde w_n(s)\|^2_{H^1(D)} \;ds \rightarrow 0
\end{displaymath}
uniformly with respect to $t \in [0,T]$. Indeed, by \eqref{eq:dNDir} 
\begin{displaymath}
 \int_0^t \|\tilde u_n(s) - \tilde w_n(s)\|^2_{H^1(D)} \;ds \leq \alpha^{-1} d_n(T)
\end{displaymath}
for all $n \in \mathbb N$ and for all $t \in [0,T]$. Moreover, it is clear that
\eqref{eq:wUnsWnDir}, \eqref{eq:sWnWnDir} and \eqref{eq:limUnPlusAnDir} hold 
uniformly on $[0,T]$. It remains to show uniform convergence of \eqref{eq:limUnNormDir}.

Fix $s \in [0,T]$. For $\epsilon > 0$ arbitrary, we choose a compact set $K \subset \Omega$
such that $\|u(s)\|_{L^2(\Omega \backslash K)} \leq \epsilon/2$. 
Since $u \in C([0,T],L^2(\Omega))$, there exists $\eta > 0$ only depending on $\epsilon$
such that $\|u(t) - u(s) \|_{L^2(\Omega)} \leq \epsilon/2$ for all 
$t \in (s-\eta, s+\eta) \cap [0,T]$. It follows that
\begin{equation}
\label{eq:normUOutsideK}
 \|u(t)\|_{L^2(\Omega \backslash K)} 
\leq \|u(t) - u(s) \|_{L^2(\Omega)} + \|u(s)\|_{L^2(\Omega \backslash K)}
\leq \epsilon/2 + \epsilon/2 = \epsilon,
\end{equation}
for all $t \in (s-\eta, s+\eta) \cap [0,T]$.
We next choose a cut-off function $\phi \in C^{\infty}_0(\Omega)$ such that 
$0 \leq \phi \leq 1$ and $\phi =1$ on $K$.  
Since $H^1_0(\Omega_n)$ converges to
$H^1_0(\Omega)$ in the sense of Mosco, we have
from Lemma \ref{lem:M1capH10} that
$ \text{cap}(\text{supp}(\phi) \cap \Omega_n^c) \rightarrow 0$. 
By definition of capacity, there exists a sequence $\xi_n \in C^{\infty}_0(\Omega)$
such that $0 \leq \xi_n \leq 1$, $\xi_n = 1$ on a neighborhood of 
$\text{supp}(\phi) \cap \Omega_n^c$ and 
$\| \xi_n\|_{H^1(\mathbb R^N)} \leq \text{cap}(\text{supp}(\phi) \cap \Omega_n^c) + 1/n$.
Define $\phi_n := (1-\xi_n) \phi$. We have that $\phi_n \in C^{\infty}_0(\Omega_n)$ and
$\phi_n \rightarrow \phi$ in $L^2(D)$.
Consider 
\begin{equation}
\label{eq:unifUnUExpandDir}
\begin{aligned}
&\left |(\tilde u_n(t) -\tilde u(t) |\tilde u(t))_{L^2(D)} \right | \\
&\leq \left | (\tilde u_n(t)|\phi_n \tilde u(t))_{L^2(D)} - (\tilde
  u(t)| \phi \tilde u(t))_{L^2(D)} \right | \\
& \quad + \left | (\tilde u_n(t)|(1-\phi_n) \tilde u(t))_{L^2(D)} -
  (\tilde u(t)| (1-\phi) \tilde u(t))_{L^2(D)} \right |  \\
& \leq \left | (\tilde u_n(t)|\phi_n \tilde u(t))_{L^2(D)} - (\tilde
  u(t)| \phi \tilde u(t))_{L^2(D)} \right |  \\
& \quad  + \left | (\tilde u_n(t) -\tilde u(t) |(1-\phi_n) \tilde u(t))_{L^2(D)} \right | 
   + \left |(\tilde u(t)| (\phi - \phi_n) \tilde u(t))_{L^2(D)} \right |.\\
\end{aligned}
\end{equation}
We prove that each term on the right of \eqref{eq:unifUnUExpandDir} is uniformly small for
$t \in (s-\eta, s+\eta) \cap [0,T]$ if n is sufficiently large. 
For the first term, applying integration by parts formula \eqref{eq:intByParts} and 
the definition of variational solutions, we obtain  
\begin{equation}
\label{eq:firstTermUnif}
\begin{aligned}
& (\tilde u_n(t)|\phi_n \tilde u(t))_{L^2(D)}  \\
&= (u_{0,n}|\phi_n u_0)_{L^2(D)} + \int_0^t \langle u_n'(s), \phi_n \tilde u(s) \rangle \;ds  
    + \int_0^t \langle u'(s), \phi_n \tilde u_n(s) \rangle \;ds \\  
&= (u_{0,n}|\phi_n u_0)_{L^2(D)} + \int_0^t \langle f_n(s), \phi_n \tilde u(s) \rangle \;ds  
    - \int_0^t a_n(s; u_n(s), \phi_n \tilde u(s)) \;ds  \\
   &\quad + \int_0^t \langle f(s), \phi_n \tilde u_n(s) \rangle \;ds 
    - \int_0^t a(s; u(s), \phi_n \tilde u_n(s)) \;ds.\\
\end{aligned}
\end{equation}
It can be easily verified using Dominated Convergence Theorem that 
$\phi_n \tilde u \rightarrow \phi \tilde u$ in $L^2((0,T),L^2(D))$. Moreover,
\begin{displaymath}
\begin{aligned}
&\int_0^T \|\phi_n \tilde u_n(t) - \phi \tilde u(t)\|^2_{L^2(D)} \;dt \\
&\leq \int_0^T \|\phi_n \tilde u(t) - \phi \tilde u(t)\|^2_{L^2(D)} \;dt
 + \int_0^T \|\phi_n \tilde u_n(t) - \phi_n \tilde u(t)\|^2_{L^2(D)} \;dt \\
&\leq \int_0^T \|\phi_n \tilde u(t) - \phi \tilde u(t)\|^2_{L^2(D)} \;dt
 + \|\phi_n\|^2_{\infty} \int_0^T \| \tilde u_n(t) -  \tilde u(t)\|^2_{L^2(D)} \;dt. \\
\end{aligned}
\end{displaymath}  
It follows also that $\phi_n \tilde u_n \rightarrow \phi \tilde u$ in $L^2((0,T),L^2(D))$.
Taking into consideration that $\tilde u_{0,n} \rightarrow \tilde u_0$ and 
$\tilde f_n \rightarrow \tilde f$, we conclude form \eqref{eq:firstTermUnif} that 
\begin{equation}
\label{eq:firstTUnif2}
(\tilde u_n(t)|\phi_n \tilde u(t))_{L^2(D)} \rightarrow (\tilde u(t)|\phi \tilde u(t))_{L^2(D)}
\end{equation}
uniformly with respect to $t \in [0,T]$. 
For the last term on the right of \eqref{eq:unifUnUExpandDir}, 
applying a similar argument as above, we write
\begin{displaymath}
 (\tilde u(t)|\phi_n \tilde u(t))_{L^2(D)}  
= (u_0|\phi_n u_0)_{L^2(D)} + 2 \int_0^t \langle u'(s), \phi_n \tilde u(s) \rangle \;ds.  
\end{displaymath}
We conclude that
\begin{equation}
\label{eq:lastTUnif}
(\tilde u(t)|\phi_n \tilde u(t))_{L^2(D)} \rightarrow (\tilde u(t)|\phi \tilde u(t))_{L^2(D)}
\end{equation} 
uniformly with respect to $t \in [0,T]$. 
Finally, for the second term on the right of \eqref{eq:unifUnUExpandDir}, we notice that
$0 \leq 1- \phi_n \leq 1$ on $\Omega$ and $1- \phi_n  = 1-(1-\xi_n)\phi
= \xi_n$ on $K$.
Moreover, using \eqref{eq:boundOfSol} and the assumption that
$\tilde u_{0,n} \rightarrow \tilde u_0$ and $\tilde f_n \rightarrow \tilde f$, 
there exists a constant $M_0 > 0$ such that
\begin{equation}
\label{eq:constM0}
 \|\tilde u_n(t)\|_{L^2(D)}, \|\tilde u(t)\|_{L^2(D)} \leq M_0,
\end{equation}
for all $t \in [0,T]$.
Hence, by Cauchy-Schwarz  inequality and \eqref{eq:normUOutsideK},
\begin{equation}
\label{eq:secondTUnif1}
\begin{aligned}
 \left | (\tilde u_n(t) -\tilde u(t) |(1-\phi_n) \tilde u(t))_{L^2(D)} \right | 
 &\leq \|\tilde u_n(t) - \tilde u(t)\|_{L^2(D)} \|(1- \phi_n) \tilde u(t)\|_{L^2(D)} \\
 &\leq 2M_0 \left ( \|u(t)\|^2_{L^2(\Omega \backslash K)} + \|\xi_n u(t) \|_{L^2(K)} \right ) \\
 &\leq 2M_0 \left ( \epsilon + \|\xi_n u(t) \|_{L^2(K)} \right ), \\
\end{aligned}
\end{equation}
for all $t \in (s-\eta, s+\eta) \cap [0,T]$ and for all $n \in \mathbb N$.
Since $\xi_n \rightarrow 0$ in $L^2(D)$, a standard argument using Dominated Convergence 
Theorem implies that $\xi_n u(s) \rightarrow 0$ in $L^2(\Omega)$. Hence,
there exists $N_{s,\epsilon} \in \mathbb N$ such that $\|\xi_n u(s)\|_{L^2(\Omega)} \leq \epsilon/2$ for 
all $n \geq N_{s,\epsilon}$. Therefore,
\begin{displaymath}
\begin{aligned}
 \|\xi_n u(t)\|_{L^2(K)} 
 &\leq \|\xi_n u(s)\|_{L^2(K)} + \| \xi_n u(t) - \xi_n u(s)\|_{L^2(K)} \\
 &\leq \|\xi_n u(s)\|_{L^2(K)} + \| u(t) -  u(s)\|_{L^2(K)}   \\
 &\leq \epsilon/2 + \epsilon/2 = \epsilon,
\end{aligned}
\end{displaymath} 
 for all $t \in (s-\eta, s+\eta) \cap [0,T]$ and  for all $n \geq N_{s,\epsilon}$.
 It follows from \eqref{eq:secondTUnif1} that 
 \begin{equation}
 \label{eq:secondTUnif2}
 \left | (\tilde u_n(t) -\tilde u(t) |(1-\phi_n) \tilde u(t))_{L^2(D)} \right | 
 \leq 2M_0  ( \epsilon + \epsilon) 
 = 4M_0 \epsilon, 
 \end{equation}
 for all $t \in (s-\eta, s+\eta) \cap [0,T]$ and  for all $n \geq N_{s,\epsilon}$.
 Therefore, by \eqref{eq:unifUnUExpandDir}, \eqref{eq:firstTUnif2}, \eqref{eq:lastTUnif}, and
 \eqref{eq:secondTUnif2}, we conclude that there exist $\tilde N_{s,\epsilon} \in \mathbb N$ 
 and a positive constant $C$ such that
 \begin{displaymath}
 \left |(\tilde u_n(t) -\tilde u(t)|\tilde u(t))_{L^2(D)} \right | \leq C \epsilon,
 \end{displaymath}
 for all $t \in (s-\eta, s+\eta) \cap [0,T]$ and for all $n \geq \tilde N_{s,\epsilon}$.  

Finally, as $[0,T]$ is a compact interval and $\eta$ only depends on $\epsilon$, it follows that
$(\tilde u_n(t)|\tilde u(t))_{L^2(D)} \rightarrow (\tilde u(t)|\tilde
u(t))_{L^2(D)}$  uniformly with respect to
$t \in [0,T]$. 
 \end{proof}
\begin{remark}
\label{rem:NoNeedStrongConvDir} 
 In fact, we can conclude that $\tilde u_n \rightarrow \tilde u$ in $L^2((0,T),L^2(D))$  
 directly from the weak convergence of solutions in Theorem \ref{th:weakConvSolDir} 
 and the compactness result in \cite[Lemma 2.1]{MR1404388}. This means we only require 
 that $\tilde f_n |_{\Omega} \rightharpoonup f$ weakly in
 $L^2((0,T),L^2(\Omega))$ and $\tilde u_{0,n} |_{\Omega} \rightharpoonup  u_0$ weakly in $L^2(\Omega)$.
 Under the same assumptions, we can restate convergence result in Theorem \ref{th:unifConvSolDir} 
 as $\tilde u_n \rightarrow \tilde u$ in $C([\delta,T],L^2(D))$ for all $\delta \in (0,T]$, as appeared in
 \cite{MR1404388}. The reason we impose stronger assumptions on the initial data and the inhomogeneous terms
 is to avoid using \cite[Lemma 2.1]{MR1404388}, which is not applicable to Neumann problems,
 and illustrate a technique that can be applied to both boundary conditions.
\end{remark}
It is known that stability under domain perturbation of solution of elliptic equations subject to
Dirichlet boundary condition can be obtained from Mosco convergence of
$H^1_0(\Omega_n)$ to $H^1_0(\Omega)$ \cite{MR1955096}. Hence we can use the same
criterion on $\Omega_n$ and $\Omega$ to conclude the stability of
solutions of parabolic equations.  In particular, the conditions on domains given in \cite[Theorem 7.5]{MR1955096} 
implies convergence of solutions of non-autonomous parabolic equations \eqref{eq:paraOmegaN}  subject to
Dirichlet boundary condition under domain perturbation.
%
%
\subsection{Neumann problems}
\label{subsec:neuProb}

It is more complicated for Neumann problems because the trivial
extension by zero outside the domain of a function $u_n$ in $H^1(\Omega_n)$ does not
belong to $H^1(D)$. In addition, as we do not assume any smoothness of
domains, there is no smooth extension operator from $H^1(\Omega_n)$ to $H^1(D)$.
In order to study the limit of $u_n \in H^1(\Omega_n)$ when the domain
is perturbed, we embed the space $H^1(\Omega_n)$ into the following
space
\begin{displaymath}
 H^1(\Omega_n) \hookrightarrow L^2(D) \times L^2(D, \mathbb R^N)
\end{displaymath}
by
\begin{displaymath}
 v_n \mapsto (\tilde v_n, \tilde \nabla v_n),
\end{displaymath}
where $\tilde v_n(x) =v(x)$ if $x \in \Omega_n$ and 
$\tilde v_n(x) = 0$ if $x \in D \backslash \Omega_n$. Similarly,
$\tilde \nabla v_n(x) = \nabla v(x)$ if $x \in \Omega_n$ and 
$\tilde \nabla v_n(x) = 0$ if $x \in D \backslash \Omega_n$.
Note that $\tilde \nabla v_n$ is not the gradient of $\tilde v_n$
in the sense of distribution. By a similar embedding for $H^1(\Omega)$,
we can consider Mosco convergence of
\begin{displaymath}
K_n := \{ (\tilde v_n, \tilde \nabla v_n) \in L^2(D) \times
          L^2(D,\mathbb R^N) \mid v_n \in H^1(\Omega_n) \}
\end{displaymath}
to
\begin{displaymath}
K:= \{ (\tilde v, \tilde \nabla v) \in L^2(D) \times
          L^2(D,\mathbb R^N) \mid v \in H^1(\Omega) \}
\end{displaymath}
in $V:=L^2(D) \times L^2(D,\mathbb R^N)$.
In this case, $K_n$ and  $K$ are closed subspace of $V$.
For simplicity, we use the term $H^1(\Omega_n)$ converges 
in the sense of Mosco to $H^1(\Omega)$ for $K_n$ and $K$ above.
 
When dealing with parabolic equations, we regard the space
$L^2((0,T),V)$ as 
$L^2((0,T),L^2(D)) \times L^2((0,T),L^2(D,\mathbb R^N))$
via the isomorphism between them. Hence 
\begin{displaymath}
\begin{aligned}
L^2((0,T),K_n) &\equiv \{ (\tilde w_n, \tilde \nabla w_n)  
  \mid w_n \in L^2((0,T),H^1(\Omega_n) \} \\
 &\subset  L^2((0,T),L^2(D)) \times L^2((0,T),L^2(D,\mathbb R^N)),
\end{aligned}
\end{displaymath}              
and 
\begin{displaymath}
\begin{aligned}
L^2((0,T),K) &\equiv \{ (\tilde w, \tilde \nabla w) 
  \mid w \in L^2((0,T),H^1(\Omega) \} \\
 &\subset L^2((0,T),L^2(D)) \times L^2((0,T),L^2(D,\mathbb R^N)).
\end{aligned}
\end{displaymath}

As in the case of Dirichlet problem, we apply various Mosco conditions
from Theorem \ref{th:equiMosco} to prove that the variational
solution $u_n$ converges to the variational solution $u$.

\begin{theorem}
 \label{th:weakConvSolNeu}
  Suppose $\tilde f_n \rightarrow \tilde f$ in $L^2((0,T),L^2(D))$ and $\tilde u_{0,n} \rightarrow \tilde u_0$ in $L^2(D)$.
  If $H^1(\Omega_n)$ converges to
  $H^1(\Omega)$ in the sense of Mosco, then
  $\tilde u_n$ converges weakly to $\tilde u$ in $L^2((0,T),L^2(D))$ and
  $\tilde \nabla u_n$ converges weakly to $\tilde\nabla u$ in
  $L^2((0,T),L^2(D,\mathbb R^N))$.
\end{theorem}
\begin{proof}
 By a similar argument as in the proof of Theorem \ref{th:weakConvSolDir},
 we have the uniform boundedness of  $(\tilde u_n, \tilde \nabla u_n)$  
 in $L^2((0,T),L^2(D)) \times L^2((0,T),L^2(D,\mathbb R^N))$.
 We can extract a subsequence (denoted again by $u_n$), such that
 $\tilde u_n \rightharpoonup w$ in $L^2((0,T),L^2(D))$ and
 $\tilde \nabla u_n \rightharpoonup (v_1,\ldots,v_N)$ in $L^2((0,T),L^2(D, \mathbb R^N))$.
 Mosco condition $(M2')$ (from Theorem \ref{th:equiMosco})
 implies that $w \in L^2((0,T),H^1(\Omega))$.

 To show that $w=u$, we let $\xi \in H^1(\Omega)$ and
 $\phi \in \mathscr D([0,T))$ and then use Mosco convergence of $H^1(\Omega_n)$ to
 $H^1(\Omega)$. In the same way as the proof of Theorem \ref{th:weakConvSolDir},
 we get \eqref{eq:limitIsWsolDir} holds for all $\xi \in H^1(\Omega)$ and
 all $\phi \in \mathscr D([0,T))$.
 Hence by the uniqueness of solution, $w=u$ in 
 $L^2((0,T),H^1(\Omega))$ and the whole  sequence converges.
\end{proof}
\begin{lemma}
\label{lem:L2weakLimUnNeu}
Suppose $\tilde f_n \rightarrow \tilde f$ in $L^2((0,T),L^2(D))$ and
$\tilde u_{0,n}\rightarrow \tilde u_0$ in $L^2(D)$. If $H^1(\Omega_n)$
converges to $H^1(\Omega)$ in the sense of Mosco, then
for each $t \in [0,T]$ we have $\tilde u_n(t)_{|_{\Omega}} \rightharpoonup  u(t)$ in $L^2(\Omega)$.
\end{lemma} 
\begin{proof}
We use the same argument as in the proof of Lemma
\ref{lem:L2weakLimUnDir} with Mosco convergence of
$H^1_0(\Omega_n)$ to $H^1_0(\Omega)$ replaced by Mosco convegence of 
$H^1(\Omega_n)$ to $H^1(\Omega)$
and the fact that $H^1(\Omega)$ is also dense in $L^2(\Omega)$.
\end{proof}
\begin{remark}
\label{rem:NeuWeakInitialF}
As remarked in the case of Dirichlet problems, 
we only require that $\tilde f_n |_{\Omega} \rightharpoonup f$ weakly in
$L^2((0,T),L^2(\Omega))$ and $\tilde u_{0,n}|_{\Omega} \rightharpoonup  u_0$ weakly in $L^2(\Omega)$ to obtain
the conclusion of Theorem \ref{th:weakConvSolNeu} and Lemma \ref{lem:L2weakLimUnNeu}.
\end{remark}
We next show the strong convergence.
\begin{theorem}
\label{th:strongConvSolNeu}
Suppose $\tilde f_n \rightarrow \tilde f$ in $L^2((0,T),L^2(D))$ and
$\tilde u_{0,n}\rightarrow \tilde u_0$ in $L^2(D)$. If $H^1(\Omega_n)$
converges to $H^1(\Omega)$ in the sense of Mosco, then
$\tilde u_n$ converges strongly to $\tilde u$ in $L^2((0,T),L^2(D))$ 
and $\tilde \nabla u_n$ converges strongly to $\tilde\nabla u$
in $ L^2((0,T),L^2(D,\mathbb R^N))$.
\end{theorem}
\begin{proof}
The proof is similar to the one in Theorem \ref{th:strongConvSolDir}. 
We show some details here for the sake of completeness. By Theorem
\ref{th:weakConvSolNeu}, $(\tilde u_n, \tilde \nabla u_n)$ 
converges weakly to $(\tilde u, \tilde \nabla u)$ 
in $L^2((0,T),L^2(D)) \times L^2((0,T),L^2(D,\mathbb R^2))$.
Since $u \in L^2((0,T),H^1(\Omega))$, Mosco condition $(M1')$
(from Theorem \ref{th:equiMosco}) implies that
there exists $w_n \in L^2((0,T),H^1(\Omega_n))$
such that $\tilde w_n \rightarrow \tilde u$ in $L^2((0,T),L^2(D))$ and 
$\tilde \nabla w_n \rightarrow \tilde\nabla u$ 
in $L^2((0,T),L^2(D,\mathbb R^N))$.
For $t \in [0,T]$, we consider 
\begin{equation}
\label{eq:dNNeu}
 \begin{aligned}
  d_n(t) &= \frac{1}{2} \| \tilde u_n(t) - \tilde u(t) \|^2_{L^2(D)}
          + \alpha \int_0^t \| u_n(s) -w_n(s) \|^2_{L^2(\Omega_n)} \;ds \\
  &\quad + \alpha \int_0^t \| \nabla u_n(s) - \nabla w_n(s) \|^2_{L^2(\Omega_n, \mathbb R^N)} \;ds \\
         &= \frac{1}{2} \| \tilde u_n(t) - \tilde u(t) \|^2_{L^2(D)}
          + \alpha \int_0^t \| u_n(s) -w_n(s) \|^2_{H^1(\Omega_n)} \;ds.
 \end{aligned}
\end{equation}
By \eqref{eq:biFormCoerOnN} (with $\lambda =0$), we can show that $d_n$
satisfies \eqref{eq:dNleqDir} for all $n \in \mathbb N$.
It can be easily seen from the weak convergence of $(\tilde u_n, \tilde \nabla u_n)$ and the strong
convergence of $(\tilde w_n, \tilde \nabla w_n)$ to $(\tilde u, \tilde \nabla u)$ 
in $L^2((0,T),L^2(D)) \times L^2((0,T),L^2(D,\mathbb R^N))$ that
\eqref{eq:wUnsWnDir} and \eqref{eq:sWnWnDir} also hold. By using Lemma
\ref{lem:L2weakLimUnNeu} instead of Lemma \ref{lem:L2weakLimUnDir}, we obtain \eqref{eq:limUnNormDir}. 
Finally, \eqref{eq:limUnPlusAnDir} is also valid for Neumann problem.
Hence, $d_n(t) \rightarrow 0$ for all $t \in [0,T]$. This shows pointwise
convergence of $\tilde u_n(t)$ to $\tilde u(t)$ in $L^2(D)$.
Moreover, by taking $t=T$ we get
\begin{displaymath}
\begin{aligned}
 &\int_0^T \|\tilde u_n(s)-\tilde u(s)\|^2_{L^2(D)} \;ds  \\
 & \quad \leq \int_0^T \|\tilde u_n(s)-\tilde w_n(s)\|^2_{L^2(D)} \;ds
  + \int_0^T \|\tilde w_n(s)-\tilde u(s)\|^2_{L^2(D)} \;ds \\
 & \quad \rightarrow 0, 
\end{aligned}
\end{displaymath}
and
\begin{displaymath}
\begin{aligned}
 &\int_0^T \|\tilde \nabla u_n(s)-\tilde\nabla u(s)\|^2_{L^2(D,\mathbb R^N)} \;ds  \\
 & \quad \leq \int_0^T \|\tilde \nabla u_n(s)-\tilde\nabla w_n(s)\|^2_{L^2(D,\mathbb R^N)} \;ds
  + \int_0^T \|\tilde \nabla w_n(s)-\tilde \nabla u(s)\|^2_{L^2(D,\mathbb R^N)} \;ds \\
 & \quad \rightarrow 0.
\end{aligned}
\end{displaymath}
This proves the strong convergence $\tilde u_n \rightarrow \tilde u$
in $L^2((0,T),L^2(D))$ and  $\tilde \nabla u_n \rightarrow \tilde\nabla u$
in $ L^2((0,T),L^2(D, \mathbb R^N))$.
\end{proof}
Recall that we embed the space $K = H^1(\Omega)$ in $V = L^2(D) \times L^2(D, \mathbb R^N)$.   
If $v$ is a function in  $W((0,T), H^1(\Omega), H^1(\Omega)')$, then
$v' \in L^2((0,T), H^1(\Omega)')$. It is not always true that we can embed 
$v'(t) \in V' = L^2(D) \times L^2(D, \mathbb R^N)$ a.e. $t \in (0,T)$ and claim that
$v \in W((0,T),V,V') \cap L^2((0,T),K)$. 
However, a similar argument as in the proof 
of Theorem \ref{th:equiMosco} $(i) \Rightarrow (iii)$ for Mosco condition (M1") 
gives the following result. 
\begin{lemma}
\label{lem:M1ddreplaceNeu}
Suppose that $H^1(\Omega_n)$ converges to $H^1(\Omega)$ in the sense of Mosco,
If $w \in C^{\infty}([0,T],H^1(\Omega))$ then there exists $w_n \in C^{\infty}([0,T],H^1(\Omega_n))$
such that $\tilde w_n$ converges to $\tilde w$ in $C^{\infty}([0,T],L^2(D))$.
\end{lemma}
\begin{proof}
We note that Proposition \ref{prop:convCombUdelta} gives uniform convergence of the 
approximation sequence in $V= L^2(D) \times L^2(D, \mathbb R^N)$.
The proof follows the same arguments as in the proof 
of Theorem \ref{th:equiMosco} $(i) \Rightarrow (iii)$. 
The only difference is that we assume here $w \in C^{\infty}([0,T],H^1(\Omega))$. Hence
the stretched function $w_{\delta} = w \circ S_{\delta}^{-1}$ belongs to 
$C^{\infty}([-\delta,T+\delta],H^1(\Omega))$.  We point out that, by using 
uniform continuity of the $k$-th order derivative $w^{(k)}$ on $[0,T]$,   
the restriction of $w_{\delta}$ on $[0,T]$ converges to $w$ in $C^{\infty}([0,T],H^1(\Omega))$.
This gives the required convergence in $C^{\infty}([0,T],L^2(D))$.
\end{proof}
  
Using the above lemma, we show in the next theorem that the solution $u_n$ of \eqref{eq:paraAbsOmegaN} indeed converges
uniformly with respect to $t \in [0,T]$.
\begin{theorem}
\label{th:unifConvSolNeu}
Suppose $\tilde f_n \rightarrow \tilde f$ in $L^2((0,T),L^2(D))$ and
$\tilde u_{0,n}\rightarrow \tilde u_0$ in $L^2(D)$. If $H^1(\Omega_n)$
converges to $H^1(\Omega)$ in the sense of Mosco, then
$\tilde u_n$ converges strongly to $\tilde u$ in $C([0,T],L^2(D))$.
\end{theorem}
\begin{proof}
As in the proof of Theorem \ref{th:unifConvSolDir}, it requires to show uniform
convergence of $(\tilde u_n(t)|\tilde u(t))_{L^2(D)} \rightarrow (\tilde
u(t)|\tilde u(t))_{L^2(D)}$.

Let $\epsilon > 0$ arbitrary. By a similar argument as in the proof Lemma \ref{lem:densityOfWK},
we have the density of $C^{\infty}([0,T], H^1(\Omega))$ in $W((0,T),H^1(\Omega), H^1(\Omega)')$.
Since the solution $u$ is in the space $W((0,T),H^1(\Omega), H^1(\Omega)')$, there exists 
$w \in C^{\infty}([0,T], H^1(\Omega))$ such that
\begin{displaymath}
\|w - u\|_{W((0,T),H^1(\Omega), H^1(\Omega)')} \leq \varepsilon.
\end{displaymath} 
As $W((0,T),H^1(\Omega), H^1(\Omega)')$ is 
continuously embedded in $C([0,T], L^2(\Omega))$, we can indeed choose 
$w \in C^{\infty}([0,T], H^1(\Omega))$ such that
\begin{equation}
\label{eq:wuLessEps}
\|w(t) - u(t)\|_{L^2(\Omega)} \leq \varepsilon,
\end{equation}
for all $t \in [0,T]$. 
By Lemma \ref{lem:M1ddreplaceNeu}, there exists $w_n \in C^{\infty}([0,T], H^1(\Omega_n))$
such that $\tilde w_n \rightarrow \tilde w$ in $C^{\infty}([0,T],L^2(D))$. 
We can write
\begin{equation}
\label{eq:unifNeu3T}
\begin{aligned}
 \left | (\tilde u_n(t)- \tilde u(t)|\tilde u(t))_{L^2(D)} \right | 
 &\leq \left | (\tilde u_n(t)|\tilde w_n(t))_{L^2(D)} - (\tilde u(t)|\tilde w(t))_{L^2(D)} \right | \\
 &\quad + \left | (\tilde u_n(t) - \tilde u(t)|\tilde u(t) -\tilde w(t))_{L^2(D)} \right | \\
 &\quad + \left | (\tilde u_n(t)|\tilde w(t) - \tilde w_n(t))_{L^2(D)} \right |, \\
\end{aligned}
\end{equation}
for all $n \in \mathbb N$.
Since $u_n$ is a solution of \eqref{eq:paraAbsOmegaN}, 
\begin{displaymath}
\begin{aligned}
 (\tilde u_n(t) |\tilde  w_n(t))_{L^2(D)} 
   &=  (\tilde u_{0,n} |\tilde w_n(0))_{L^2(D)}
    + \int_0^t \langle f_n(s), w_n(s) \rangle \;ds  \\
    &\quad + \int_0^t \langle w_n'(s), u_n(s) \rangle \;ds
    - \int_0^t a_n(s; u_n(s),w_n(s)) \rangle \;ds, \\
\end{aligned}
\end{displaymath}
for all $n \in \mathbb N$.
Taking into consideration that $\tilde u_{0,n} \rightarrow \tilde u_0$ and 
$f_n \rightarrow f$, we conclude that 
\begin{equation}
\label{eq:firstTNeu}
(\tilde u_n(t)|\tilde w_n(t))_{L^2(D)} \rightarrow (\tilde u(t)|\tilde w(t))_{L^2(D)}
\end{equation}
uniformly with respect to $t \in [0,T]$. Moreover, by
\eqref{eq:wuLessEps} and the uniform boundedness of solutions as in
\eqref{eq:constM0},
\begin{equation}
\label{eq:secondTNeu}
\begin{aligned}
&\left | (\tilde u_n(t) - \tilde u(t)|\tilde u(t) -\tilde w(t))_{L^2(D)} \right |  \\
&\leq \| \tilde u_n(t) - \tilde u(t)\|_{L^2(D)} \|\tilde u(t) -\tilde w(t)\|_{L^2(D)} \\
&\leq 2M_0 \epsilon, \\
\end{aligned}
\end{equation}
for all $t \in [0,T]$ and for all $n \in \mathbb N$. 
Finally, as $\tilde w_n \rightarrow \tilde w$ in
$C^{\infty}([0,T],L^2(D))$, there exists $N_{\epsilon} \in \mathbb N$ such that
\begin{displaymath}
\| \tilde w_n(t) - \tilde w(t) \|_{L^2(D)} \leq \epsilon,
\end{displaymath}
for all $t \in [0,T]$ and for all $n \geq N_{\epsilon}$. Hence,
\begin{equation}
\label{eq:lastTNeu}
\begin{aligned}
 \left | (\tilde u_n(t)|\tilde w_n(t) -\tilde w(t))_{L^2(D)} \right |
  &\leq \|\tilde u_n(t)\|_{L^2(D)} \| \| \tilde w_n(t) -\tilde w(t)\|_{L^2(D)} \\
  &\leq M_0 \epsilon, \\
 \end{aligned}
\end{equation}
for all $t \in [0,T]$ and for all $n \geq N_{\epsilon}$.
Therefore, by \eqref{eq:unifNeu3T} -- \eqref{eq:lastTNeu},
there exists $\tilde N_{\epsilon} \in \mathbb N$ 
and a positive constant $C$ such that 
\begin{displaymath}
 \left |(\tilde u_n(t) -\tilde u(t)|\tilde u(t))_{L^2(D)} \right | \leq C \epsilon,
 \end{displaymath}
 for all $t \in [0,T]$ and for all $n \geq \tilde N_{\epsilon}$. As $\epsilon >0$
 was arbitrary, this proves the required uniform convergence of 
 $(\tilde u_n(t)|\tilde u(t))_{L^2(D)} \rightarrow (\tilde u(t)|\tilde
 u(t))_{L^2(D)}$ with respect to
 $t \in [0,T]$.
\end{proof} 
We can use the same criterion on $\Omega_n$ and $\Omega$ as in Neumann elliptic
problems to conclude the stability of solutions of Neumann parabolic equations
under domain perturbation. 
In particular, for domains in two dimensional spaces, the conditions on domains given in \cite[Theorem 3.1]{MR1822408} implies
convergence of solutions of non-autonomous parabolic equations \eqref{eq:paraOmegaN} subject to Neumann boundary condition.
%
%
%
\begin{remark}
\label{rem:domainCrackNeu} 
 The assumptions on strong convergence of $\tilde u_{0,n}$ and $\tilde f_n$
 can be weaken if we impose some regularity of the domains.
 We give an example of domains $\Omega_n$ satisfying the \emph{cone condition} (see \cite[Section 4.3]{MR0450957}) 
 uniformly with respect to $n \in \mathbb N$.
   
 Let $N=2$ and let 
 \begin{displaymath}
  \begin{aligned}
   \Omega &:= \{ x \in \mathbb R^2 : |x| <1 \} \backslash \{(x_1,0): 0 \leq x_1 <1 \}, \\ 
   \Omega_n &:= \{ x \in \mathbb R^2 : |x| <1 \} \backslash \{(x_1,0): \delta_n \leq x_1 <1 \}, 
  \end{aligned}
 \end{displaymath}
 where $\delta_n \searrow 0$.  
 This example is an exterior perturbation of the domain, that is
 $\Omega \subset \Omega_{n+1} \subset \Omega_n$ for all $n \in \mathbb N$.
 It is easy to see that $\Omega$ and $\Omega_n$  satisfy the cone condition
 uniformly with respect to $n \in \mathbb N$, but $H^1(\Omega)$ and $H^1(\Omega_n)$ do not have
 the extension property.
 Moreover, these domains satisfy the conditions in \cite{MR1822408}. 
 Hence, $H^1(\Omega_n)$ converges to $H^1(\Omega)$ in the sense of Mosco. 
 Note that here we take $D$ to be the open unit disk center at $0$ in $\mathbb R^2$.
 In this example, we only need that $\tilde f_n |_{\Omega} \rightharpoonup f$
 in $L^2((0,T),L^2(\Omega))$ and $\tilde u_{0,n} |_{\Omega} \rightharpoonup u_0$ in $L^2(\Omega)$ to conclude the 
 convergence of solutions $\tilde u_n \rightarrow \tilde u$ in $C([\delta, T], L^2(D))$ for all $\delta \in (0,T]$.
 In addition, if $\tilde u_{0,n} \rightarrow \tilde u_0$ in $L^2(D)$, then the assertion holds for $\delta = 0$.

  To see this, we note from Lemma \ref{lem:L2weakLimUnNeu} 
  (taking Remark \ref{rem:NeuWeakInitialF} into account) that $\tilde u_n(t)_{|_{\Omega}} \rightharpoonup u(t)$
  in $L^2(\Omega)$ weakly for all $t \in [0,T]$. Since $u \in L^2((0,T),H^1(\Omega))$, 
  we have $u(t) \in H^1(\Omega)$ for almost everywhere $t \in (0,T)$.
  Fix now such $t \in (0,T)$. By the continuity of the solutions $u_n \in C([0,T],L^2(\Omega_n))$, for each $n \in \mathbb N$
  we can choose $\rho_n > 0$ such that 
  \begin{displaymath}
    \|u_n(s) -u_n(t) \|_{L^2(\Omega_n)} \leq \frac{1}{n}
  \end{displaymath} 
  for all $s \in (t-\rho_n, t+\rho_n) \cap (0,T)$.  As $u_n \in L^2((0,T),H^1(\Omega_n))$ we can choose 
  $t_n \in (t-\rho_n, t+\rho_n) \cap (0,T)$ such that $u_n(t_n) \in H^1(\Omega_n)$ for all $n \in \mathbb N$.
  For these choices of $t_n$, we have $ \|u_n(t_n) -u_n(t) \|_{L^2(\Omega_n)} \rightarrow 0$ as $n \rightarrow \infty$.
  It follows that $\tilde u_n(t_n)_{|_{\Omega}} \rightharpoonup u(t)$ in $L^2(\Omega)$ weakly.
  Since $\Omega \subset \Omega_n$ for all $n \in \mathbb N$, the restriction 
  $ u_{n}(t_n)_{|_{\Omega}}$  belongs to $H^1(\Omega)$ for all $n \in \mathbb N$. 
  Hence it follows from the weak convergence of 
  $\tilde u_n(t_n)_{|_{\Omega}} =  u_n(t_n)_{|_{\Omega}} \rightharpoonup u(t)$ in $L^2(\Omega)$ 
  that
  \begin{displaymath}
    \int_{\Omega}  \partial_j \big( u_{n}(t_n)_{|_{\Omega}} \big )  \phi dx 
   =  -\int_{\Omega} u_{n}(t_n)_{|_{\Omega}} \partial_j \phi  dx
   \rightarrow -\int_{\Omega} u(t) \partial_j \phi  dx
    =  \int_{\Omega} \partial_j u (t) \phi dx,  
  \end{displaymath}
  for all $\phi \in C^{\infty}_c(\Omega)$ for $j =1,2$. 
  This means $\nabla u_{n}(t_n)_{|_{\Omega}} \rightharpoonup \nabla u(t)$
  in $L^2(\Omega, \mathbb R^2)$. Thus, $ u_{n}(t_n)_{|_{\Omega}}$ is bounded in $H^1(\Omega)$. 
  As $\Omega$ is bounded and satisfies 
  the cone condition, we have from the Rellich-Kondrachov theorem 
  that the embedding $H^1(\Omega) \hookrightarrow L^2(\Omega)$ is compact (see \cite[Theorem 6.2]{MR0450957}).
  Therefore, $u_{n}(t_n)_{|_{\Omega}}$ has a subsequence which converges strongly in $L^2(\Omega)$. 
  Since we have a prior 
  knowledge of weak convergence $u_n(t_n)_{|_{\Omega}} \rightharpoonup u(t)$ in $L^2(\Omega)$, we conclude that the whole 
  sequence  $u_n(t_n)_{|_{\Omega}} \rightarrow u(t)$ in $L^2(\Omega)$ strongly. 
  By the choices of $t_n$, we conclude that $u_n(t)_{|_{\Omega}} \rightarrow u(t)$ in $L^2(\Omega)$ strongly.
  Since the above argument works for almost everywhere 
  $t \in (0,T)$, we deduce from the dominated convergence theorem that
  $\tilde u_n |_{\Omega} = u_n |_{\Omega} \rightarrow u$ in $L^2((0,T),L^2(\Omega))$ strongly. 
  As the cutting line is a set of measure zero in $\mathbb R^2$, 
  we have $\tilde u_n  \rightarrow \tilde u$ in $L^2((0,T),L^2(D))$. 
  By extracting a subsequence (indexed again by $n$), 
  we can choose $\delta$ arbitrarily closed to zero in $(0,T]$
  such that $\tilde u_n(\delta) \rightarrow \tilde u(\delta)$ in $L^2(D)$. 
  The required convergence in $C([\delta,T],L^2(D))$ follows   
  from the argument in the proof of Theorem \ref{th:strongConvSolNeu} and Theorem \ref{th:unifConvSolNeu}
  with the integration taken over $[\delta,T]$ instead of $[0,T]$ (see also the proof of \cite[Therem 3.5]{MR1404388}). 
\end{remark}

\section{Application in Parabolic Variational Inequalities}
\label{sec:appParVar}

In the previous section we have seen some applications of Theorem \ref{th:equiMosco}
when $K_n$ and $K$ are closed subspaces of $V$. In this section, we consider
the case when $K_n$ and $K$ are just closed and 
convex subsets of $V$. We show here a similar convergence properties
of solutions of parabolic variational inequalities.

Let $K_n, K$ be closed and convex subsets in $V$. 
For each $t \in (0,T)$, suppose $a(t; \cdot, \cdot)$ is a continuous 
bilinear form on $V$ satisfying \eqref{eq:biFormCont} and \eqref{eq:biFormCoer}.
For simplicity, we assume that $\lambda = 0$ in \eqref{eq:biFormCoer}.
We denote by $A(t)$ the linear operator induced by $a(t; \cdot, \cdot)$.
Let us consider the following parabolic variational inequalities.
Given $u_{0,n} \in K_n$ and $f_n \in L^2((0,T),V')$, we want to find $u_n$ such that
for a.e. $t \in (0,T)$, $u_n(t) \in K_n$ and 
\begin{equation}
\label{eq:parVarIneqKn}
 \left\{ 
  \begin{aligned}
   \langle u'(t), v-u(t) \rangle  + \langle A(t)u(t), v-u(t) \rangle 
     - \langle f_n(t) , v - u(t) \rangle &\geq 0, \quad \forall v \in K_n \\
   u(0) &= u_{0,n}.
   \end{aligned}
  \right .
\end{equation}
When $K_n, f_n$ and $u_{0,n}$ converge to $K, f$ and $u_{0}$, 
we wish to obtain convergence results of weak solution of 
\eqref{eq:parVarIneqKn} to the following limit inequalities.
\begin{equation}
\label{eq:parVarIneqKK}
 \left\{ 
  \begin{aligned}
   \langle u'(t), v-u(t) \rangle  + \langle A(t)u(t), v-u(t) \rangle 
     - \langle f(t) , v - u(t) \rangle &\geq 0, \quad \forall v \in K \\
   u(0) &= u_{0}.
   \end{aligned}
  \right .
\end{equation}
Throughout this section, we denote the weak solution of \eqref{eq:parVarIneqKn}
by $u_n$ and the weak solution of \eqref{eq:parVarIneqKK} by $u$. The notion of
our weak solutions is given in Definition \ref{def:weakSolIneq}.

\begin{theorem}
\label{th:boundedParVar}
Suppose $f_n \rightarrow f$ in $L^2((0,T),V')$, $u_{0,n} \rightharpoonup u_0$ in $V$ 
and $u_{0,n} \rightarrow  u_0$ in $H$. 
Then the sequence of weak solutions $u_n$ is bounded in $L^2((0,T),V)$.
\end{theorem}
\begin{proof}
Let $v \in W((0,T),V,V') \cap L^2((0,T),K)$ be the constant function
defined by $v(t) := u_0$ for $t \in [0,T]$. Similarly, 
$v_n \in W((0,T),V,V') \cap L^2((0,T),K_n)$ defined by 
$v_n(t) := u_{0,n}$ for $t \in [0,T]$. It follows that $v_n
\rightharpoonup v$ in $L^2((0,T),V)$. 
Since $u_n$ is a weak solution of \eqref{eq:parVarIneqKn},
\begin{displaymath}
\begin{aligned}
    &\int_0^T   \langle A(t)u_n(t), u_n(t)-v_n(t) \rangle \;dt \\
    & \quad \leq \int_0^T \langle v_n'(t), v_n(t)-u_n(t) \rangle
     - \langle f_n(t), v_n(t)-u_n(t) \rangle  \;dt 
     + \frac{1}{2} \|v_n(0) -u_{0,n} \|^2_H  \\
    & \quad = - \int_0^T \langle f_n(t), v_n(t)-u_n(t) \rangle  \;dt.
\end{aligned}
\end{displaymath}
Thus, 
\begin{displaymath}
\begin{aligned}
    &\int_0^T   \langle A(t)u_n(t) - A(t)v_n(t), u_n(t)-v_n(t) \rangle \;dt \\
    &\quad \leq \int_0^T  \langle A(t)v_n(t), v_n(t) -u_n(t) \rangle
     - \langle f_n(t), v_n(t)-u_n(t) \rangle  \;dt \\
    &\quad \leq \| A(t)v_n -f_n \|_{L^2((0,T),V')} \|v_n -u_n\|_{L^2((0,T),V)}.
\end{aligned}
\end{displaymath}
By the coerciveness of $A(t)$, 
\begin{displaymath}
 \alpha \|u_n -v_n\|_{L^2((0,T),V)} \leq \|A(t)v_n -f_n \|_{L^2((0,T),V')}.
\end{displaymath} 
We conclude from the weak convergences of $v_n$ and $f_n$ that $u_n$ is
bounded in $L^2((0,T),V)$.
\end{proof}
\begin{theorem}
\label{th:weakConvSolParVar}
Suppose $f_n \rightarrow f$ in $L^2((0,T),V')$, $u_{0,n} \rightharpoonup u_0$ in $V$ 
and $u_{0,n} \rightarrow  u_0$ in $H$. If $K_n$ converges to $K$ in the sense of Mosco,
then the sequence of weak solutions $u_n$ converges weakly to $u$ in $L^2((0,T),V)$.  
\end{theorem}
\begin{proof}
By Theorem \ref{th:boundedParVar}, we can extract a subsequence of $u_n$
(denoted again by $u_n$) such that $u_n \rightharpoonup \kappa$ in
$L^2((0,T),V)$. Since $u_n \in L^2((0,T),K_n)$, we
apply Mosco condition $(M2')$ (from Theorem \ref{th:equiMosco}) 
to deduce that the weak limit $\kappa \in L^2((0,T),K)$.
By the uniqueness of weak solution, it suffices to prove
that $\kappa$ satisfies \eqref{eq:weakSolParVar} (with $u$ replaced by
$\kappa$) in the definition of weak solution.

By Mosco condition $(M1')$, there exists $w_n \in L^2((0,T),K_n)$ such
that $w_n \rightarrow \kappa$ in $L^2((0,T),V)$. Let $v \in W((0,T),V,V')
\cap L^2((0,T),K)$. We again apply Theorem \ref{th:equiMosco} for Mosco
condition $(M1")$ to get a sequence of functions
$v \in W((0,T),V,V')\cap L^2((0,T),K_n)$ such that 
$v_n \rightarrow v$ in $W((0,T),V,V')$. For each $n \in \mathbb N$,
\begin{displaymath}
\begin{aligned}
 &\langle A(t)w_n(t), v_n(t)-u_n(t) \rangle \\
 &\quad= \langle A(t)u_n(t), v_n(t)-u_n(t) \rangle  
   + \langle A(t)w_n(t) - A(t)u_n(t), v_n(t) -u_n(t) \rangle \\
 &\quad= \langle A(t)u_n(t), v_n(t)-u_n(t) \rangle  
   + \langle A(t)w_n(t) - A(t)u_n(t), w_n(t) -u_n(t) \rangle \\
   &\quad \quad + \langle A(t)w_n(t) - A(t)u_n(t), v_n(t) -w_n(t) \rangle.
\end{aligned}
\end{displaymath}
Hence, by definition of weak solution on $K_n$ and coerciveness of $A(t)$,
\begin{displaymath}
\begin{aligned}
 &\int_0^T \langle v_n'(t), v_n(t)-u_n(t) \rangle 
  + \langle A(t)w_n(t), v_n(t)-u_n(t) \rangle \;dt  \\
 &\quad - \int_0^T \langle f_n(t), v_n(t)-u_n(t) \rangle \;dt 
  + \frac{1}{2} \|v_n(0) - u_{0,n}\|^2_H  \\
 &= \int_0^T \langle v_n'(t), v_n(t)-u_n(t) \rangle 
  + \langle A(t)u_n(t), v_n(t)-u_n(t) \rangle \;dt  \\
 &\quad - \int_0^T \langle f_n(t), v_n(t)-u_n(t) \rangle \;dt 
  + \frac{1}{2} \|v_n(0) - u_{0,n}\|^2_H  \\
 &\quad + \int_0^T \langle A(t)w_n(t) - A(t)u_n(t), w_n(t) -u_n(t) \rangle \;dt \\
 & \quad + \int_0^T \langle A(t)w_n(t) - A(t)u_n(t), v_n(t) -w_n(t) \rangle \;dt \\
 & \geq \int_0^T \langle A(t)w_n(t) - A(t)u_n(t), v_n(t) -w_n(t) \rangle \;dt,
\end{aligned}
\end{displaymath}
for all $n \in \mathbb N$. Letting $n \rightarrow \infty$,
\begin{displaymath}
\begin{aligned}
&\int_0^T \langle v'(t), v(t)-\kappa(t) \rangle 
  + \langle A(t)\kappa(t), v(t)-\kappa(t) \rangle \;dt  \\
 &\quad - \int_0^T \langle f(t), v(t)-\kappa(t) \rangle \;dt 
  + \frac{1}{2} \|v(0) - u_0\|^2_H 
  \geq 0.
\end{aligned}
\end{displaymath}
This implies $\kappa$ is a weak solution of \eqref{eq:parVarIneqKK} 
as required.
\end{proof}
We finally prove strong convergence of solutions.
\begin{theorem}
\label{th:strongConvSolParVar}
Suppose $f_n \rightarrow f$ in $L^2((0,T),V')$, $u_{0,n} \rightharpoonup u_0$ in $V$ 
and $u_{0,n} \rightarrow  u_0$ in $H$. If $K_n$ converges to $K$ in the sense of Mosco,
then the sequence of weak solutions $u_n$ converges strongly to $u$ in $L^2((0,T),V)$. 
\end{theorem}
\begin{proof}
By the coerciveness of $A(t)$, 
\begin{equation}
\label{eq:limInfSol}
 \liminf_{n \rightarrow \infty}
 \int_0^T \langle A(t)u_n(t) - A(t)u(t), u_n(t) -u(t) \rangle \;dt \geq  0. 
\end{equation}
For each $\epsilon > 0$, we define $u_{\epsilon}$ by 
\begin{displaymath}
 \begin{aligned}
 \epsilon u'_{\epsilon} + u_{\epsilon}& = u \\
                      u_{\epsilon} (0)& = u_0.
 \end{aligned}
\end{displaymath}
Then $u_{\epsilon} \in W((0,T),V,V') \cap L^2((0,T),K)$ and
$u_{\epsilon} \rightarrow u$ in $L^2((0,T),V)$ as $\epsilon \rightarrow
0$ (see in the proof of \cite[Theorem 2.3]{MR0296479}). 
For each $\epsilon > 0$, Mosco condition $(M1")$ 
(from Theorem \ref{th:equiMosco}) implies that there exists
$u_{\epsilon,n} \in W((0,T),V,V') \cap L^2((0,T),K_n)$ such that 
$u_{\epsilon,n} \rightarrow u_{\epsilon}$ in $W((0,T),V,V')$ as
$n \rightarrow \infty$. Since $u_n$ is a weak solution of \eqref{eq:parVarIneqKn},
\begin{displaymath}
\begin{aligned}
 &\int_0^T \langle A(t)u_n(t), u_n(t)-u(t) \rangle \;dt \\
 &\leq \int_0^T \langle v'_n(t),v_n(t)- u_n(t) \rangle \;dt
  -  \int_0^T \langle f_n(t),v_n(t)- u_n(t) \rangle \;dt \\
 &\quad + \frac{1}{2} \|v_n(0) - u_{0,n} \|^2_H 
  + \int_0^T \langle A(t)u_n(t),v_n(t)- u(t) \rangle \;dt,
\end{aligned}
\end{displaymath}
for all $v \in W((0,T),V,V') \cap L^2((0,T),K)$. In particular, taking
$v_n = u_{\epsilon,n}$, 
\begin{displaymath}
\begin{aligned}
 &\int_0^T \langle A(t)u_n(t), u_n(t)-u(t) \rangle \;dt \\
 &\leq \int_0^T \langle u'_{\epsilon,n}(t),u_{\epsilon,n}(t)- u_n(t) \rangle \;dt
  - \int_0^T \langle f_n(t),u_{\epsilon,n}(t)- u_n(t) \rangle \;dt \\
 &\quad + \frac{1}{2} \|u_{\epsilon,n}(0) - u_{0,n} \|^2_H 
  + \int_0^T \langle A(t)u_n(t),u_{\epsilon,n}(t)- u(t) \rangle \;dt.
\end{aligned}
\end{displaymath}
Letting $n \rightarrow \infty$, we obtain
\begin{displaymath}
\begin{aligned}
 &\limsup_{n \rightarrow \infty}
  \int_0^T \langle A(t)u_n(t), u_n(t)-u(t) \rangle \;dt \\
 &\leq \int_0^T \langle u'_{\epsilon}(t),u_{\epsilon}(t)- u(t) \rangle \;dt
  -  \int_0^T \langle f(t),u_{\epsilon}(t)- u(t) \rangle \;dt \\
 &\quad + \frac{1}{2} \|u_{\epsilon}(0) - u_0 \|^2_H 
  +  \int_0^T \langle A(t)u(t),u_{\epsilon}(t)- u(t) \rangle \;dt \\
 &= -\epsilon \int_0^T \|u'_{\epsilon}(t)\|^2_H \;dt
    -  \int_0^T \langle f(t),u_{\epsilon}(t)- u(t) \rangle \;dt \\
 &\quad + \int_0^T \langle A(t)u(t),u_{\epsilon}(t)- u(t) \rangle \;dt \\
 &\leq  - \int_0^T \langle f(t),u_{\epsilon}(t)- u(t) \rangle \;dt 
 + \int_0^T \langle A(t)u(t),u_{\epsilon}(t)- u(t) \rangle \;dt.
\end{aligned}
\end{displaymath}
This is true for any $\epsilon > 0$. Hence, by letting $\epsilon
\rightarrow 0$, 
\begin{displaymath}
 \limsup_{n \rightarrow \infty}
  \int_0^T \langle A(t)u_n(t), u_n(t)-u(t) \rangle \;dt \leq 0.
\end{displaymath}
On the other hand, the weak convergence of $u_n$ in Theorem
\ref{th:weakConvSolParVar} implies
\begin{displaymath}
 \lim_{n \rightarrow 0} 
  \int_0^T \langle A(t)u(t), u_n(t)-u(t) \rangle \;dt = 0.
\end{displaymath}
Thus,
\begin{equation}
\label{eq:limSupSol}
 \limsup_{n \rightarrow \infty}
 \int_0^T \langle A(t)u_n(t) - A(t)u(t), u_n(t) -u(t) \rangle dt \leq  0. 
\end{equation}
It follows from the coerciveness of $A(t)$, \eqref{eq:limInfSol} and
\eqref{eq:limSupSol} that
\begin{displaymath}
 \alpha \|u_n -u \|^2_{L^2((0,T),V)}  
 \leq  \int_0^T \langle A(t)u_n(t) - A(t)u(t), u_n(t) -u(t) \rangle \;dt
 \rightarrow 0,
\end{displaymath}
as $n \rightarrow \infty$. Therefore, $u_n \rightarrow u$ in $L^2((0,T),V)$.
\end{proof}
 
\section*{Acknowledgments} 
The author would like to thank D. Daners for helpful discussions and suggestions.
%

\bibliography{database}{}

\begin{thebibliography}{10}

\bibitem{MR0450957}
R.~A. Adams.
\newblock {\em Sobolev spaces}.
\newblock Academic Press, New York-London, 1975.

\bibitem{MR1832168}
W.~Arendt.
\newblock Approximation of degenerate semigroups.
\newblock {\em Taiwanese J. Math.}, 5(2):279--295, 2001.

\bibitem{MR773850}
H.~Attouch.
\newblock {\em Variational convergence for functions and operators}.
\newblock Applicable Mathematics Series. Pitman (Advanced Publishing Program),
  Boston, MA, 1984.

\bibitem{MR0428137}
H.~Br{\'e}zis.
\newblock Probl\`emes unilat\'eraux.
\newblock {\em J. Math. Pures Appl. (9)}, 51:1--168, 1972.

\bibitem{MR1822408}
D.~Bucur and N.~Varchon.
\newblock Boundary variation for a {N}eumann problem.
\newblock {\em Ann. Scuola Norm. Sup. Pisa Cl. Sci. (4)}, 29(4):807--821, 2000.

\bibitem{MR1951783}
D.~Bucur and N.~Varchon.
\newblock A duality approach for the boundary variation of {N}eumann problems.
\newblock {\em SIAM J. Math. Anal.}, 34(2):460--477 (electronic), 2002.

\bibitem{MR1995490}
G.~dal Maso, F.~Ebobisse, and M.~Ponsiglione.
\newblock A stability result for nonlinear {N}eumann problems under boundary
  variations.
\newblock {\em J. Math. Pures Appl. (9)}, 82(5):503--532, 2003.

\bibitem{MR1404388}
D.~Daners.
\newblock Domain perturbation for linear and nonlinear parabolic equations.
\newblock {\em J. Differential Equations}, 129(2):358--402, 1996.

\bibitem{MR1955096}
D.~Daners.
\newblock Dirichlet problems on varying domains.
\newblock {\em J. Differential Equations}, 188(2):591--624, 2003.

\bibitem{MR2119988}
D.~Daners.
\newblock Perturbation of semi-linear evolution equations under weak
  assumptions at initial time.
\newblock {\em J. Differential Equations}, 210(2):352--382, 2005.

\bibitem{MR1156075}
R.~Dautray and J.-L. Lions.
\newblock {\em Mathematical analysis and numerical methods for science and
  technology. {V}ol. 5}.
\newblock Springer-Verlag, Berlin, 1992.

\bibitem{MR1009162}
N.~Dunford and J.~T. Schwartz.
\newblock {\em Linear operators. {P}art {I}}.
\newblock Wiley Classics Library. John Wiley \& Sons Inc., New York, 1988.

\bibitem{MR2210083}
K.~Ito and K.~Kunisch.
\newblock Parabolic variational inequalities: the {L}agrange multiplier
  approach.
\newblock {\em J. Math. Pures Appl. (9)}, 85(3):415--449, 2006.

\bibitem{MR0259693}
J.-L. Lions.
\newblock {\em Quelques m\'ethodes de r\'esolution des probl\`emes aux limites
  non lin\'eaires}.
\newblock Dunod, 1969.

\bibitem{MR0296479}
{\v{Z}}.~L. Lions.
\newblock Partial differential inequalities.
\newblock {\em Uspehi Mat. Nauk}, 26(2(158)):205--263, 1971.

\bibitem{MR0298508}
U.~Mosco.
\newblock Convergence of convex sets and of solutions of variational
  inequalities.
\newblock {\em Advances in Math.}, 3:510--585, 1969.

\bibitem{MR1855977}
F.~Simondon.
\newblock Domain perturbation for parabolic quasilinear problems.
\newblock {\em Commun. Appl. Anal.}, 4(1):1--12, 2000.

\bibitem{MR1033497}
E.~Zeidler.
\newblock {\em Nonlinear functional analysis and its applications. {II}/{A}}.
\newblock Springer-Verlag, New York, 1990.

\end{thebibliography}
\bibliographystyle{abbrv}

%

%
%
%

\end{document}